\documentclass[a4paper,11pt]{article} 
\usepackage{inputenc} 
\usepackage{hyperref}

\usepackage[dvips]{graphicx}
\usepackage[dvipsnames,svgnames]{xcolor} 
\usepackage{tikz} 

\usepackage{amssymb,amsmath,amsfonts,amsthm}
\usepackage{enumerate}

\usepackage[shortlabels]{enumitem}

\newcommand{\changed}[1]{#1}

\newtheorem{lemma}{Lemma}[section]
\newtheorem{corollary}[lemma]{Corollary}
\newtheorem{proposition}[lemma]{Proposition}
\newtheorem{theorem}[lemma]{Theorem}
\newtheorem{remark}[lemma]{Remark}
\newtheorem{definition}[lemma]{Definition}
\newtheorem{construction}[lemma]{Construction}

\newcommand{\Xsl}{X_{s,\lambda}}
\newcommand{\Ysl}{Y_{s,\lambda}}

\newcommand{\col}[1]{\color{blue!70!}{#1} \color{black}{}}
\newcommand{\abs}[1]{\left | #1 \right | }
\newcommand{\norm}[1]{ \left\| #1 \right\| }
\newcommand{\Fcal}{ \mathcal{F} }
\newcommand{\F}{ \mathcal{F} }

\newcommand{\set}[1]{\left\{ #1 \right\}}

\newcommand{\Fsl}{\mathcal{F}_{s,\lambda}}

\newcommand{\curvin}{\frac{1-s}{\lambda}}
\newcommand{\curvout}{-\frac{s}{\lambda}}
\renewcommand{\H}{\mathcal{H}}

\newcommand{\fraceins}{\frac{\abs{S_1}}{P(S_1)}}
\newcommand{\fraczwei}{\frac{\abs{S_2}}{P(S_2)}}
\newcommand{\fracco }{\frac{\abs{S}}{P\left({\rm co}\lr{S}\right)}}

\newcommand{\open}[2]{{\rm {\it Open}}_{#1}\left(#2 \right)}
\newcommand{\close}[2]{{\rm {\it Close}}_{#1}\left(#2 \right)}

\newcommand{\fraceinsi}{\frac{P(S_1)}{\abs{S_1}}}
\newcommand{\fraczweii}{\frac{P(S_2)}{\abs{S_2}} }

\newcommand{\Cs}{{C}_{s,\lambda}}
\newcommand{\tCs}{{\tilde {C}}_{s,\lambda}}

\newcommand{\lr}[1]{\left( #1 \right)}

\newcommand{\rinnen}{\frac{\lambda}{1-s}}
\newcommand{\raussen}{\frac{\lambda}{s}}
\newcommand{\GammaS}{\Gamma_{s,\lambda}\lr{S}}

\def\Carre#1#2{\vbox{
   \hrule height .#2pt
   \hbox{\vrule width .#2pt height #1pt \kern #1pt
      \vrule width .#2pt}
   \hrule height .#2pt}}

\def\R{\mathbb{R}}

\def\N{\mathbb{N}}
\def\H{\mathcal{H}}

\def\Div{\textup{div}\,}

\title{TV denoising of two balls in the plane}
\author{
Vicent Caselles\thanks{
DTIC, Universitat Pompeu--Fabra,
Roc Boronat 138,
08018 Barcelona, Spain}, 
Matteo Novaga\thanks{
Dipartimento di Matematica, Universit\`a di Pisa, Largo Bruno Pontecorvo 5
56127 Pisa, Italy,
e--mail: novaga@dm.unipi.it
},
Christiane  P\"oschl \thanks{
Institute of Mathematics,
Alpen Adria University Klagenfurt,
Universit\"atsstra\ss{}e 65, 9020 Klagenfurt
e--mail: Christiane.Poeschl@aau.at
}}

\date{}

\begin{document}
\maketitle

\begin{abstract}
The aim of this paper is to compute the explicit solution of the total variation denoising problem
corresponding to the characteristic function of a set which is the union of two 
planar disjoint balls with different radii.
\end{abstract}

\section{Introduction}\label{sect:introduction}
The purpose of this paper is to compute explicit solutions of the total variation denoising problem
\begin{equation} \label{vpI}
\min_{u \in BV(\R^2)\cap L^2(\R^2)}
\int_{\R^2} \abs{Du} + \frac{1}{2 \lambda} \int_{\R^2} \abs{u - f }^2 \, dx \;,
\end{equation}
where $f=\chi_S$ and $S \in \R^2$ is the union of two balls whose interiors are disjoint sets and $\lambda > 0$.

The study of explicit solutions of 
\eqref{vpI} 
was initiated in \cite{BeCaNo:00,BeCaNo:05}, where the authors
studied the bounded sets of finite perimeter $S$ in $\R^2$ for which the solution of \eqref{vpI} is a multiple of $\chi_S$.
Such sets, which were called calibrable,
produce solutions of the total variation flow which evolve at constant speed without distortion of the boundary.
They where characterized in \cite{BeCaNo:00} by the existence of a vector field $\xi\in L^\infty(\R^2,\R^2)$ such that $\vert\xi\vert\leq 1$,
$\xi\cdot D\chi_S = |D\chi_S|$, and $-\mathrm{div}\, \xi = \frac{P(S)}{|S|}\chi_S$, where $P(S)$ denotes the perimeter of $S$ and
$|S|$ denotes the area of $S$.
For bounded connected sets $S\subset \R^2$, calibrable sets are characterized as the convex, $C^{1,1}$ sets satisfying the bound $\hbox{\rm ess sup}_{x\in\partial S} \kappa(x) \leq \frac{P(S)}{|S|}$, where $\kappa(x)$ denotes the curvature of $\partial S$ at the point $x$.
The paper \cite{BeCaNo:00} gives also a characterization of non connected calibrable sets in $\R^2$.
The paper \cite{BeCaNo:05} describes the explicit solution of
\eqref{vpI} for sets $S\subset \R^2$ of the form $S = C_0\setminus \cup_{i=1}^k C_i $
where the sets $C_i$, $i=0,1,\ldots,k$, are convex and satisfy some bounds on the curvature of its boundary.
The explicit solution when the set $S$ is a convex subset of $\R^2$ was described in
\cite{Alter.Caselles.Chambolle05b} (also the case of a set $S$ which is a union of convex sets which are
sufficiently far from each other). The explicit solution corresponding to a general convex set in $\R^N$ was described in the papers
\cite{Alter.Caselles.Chambolle05,CCN:06,AC} (covering also the case of the union
of convex sets which are sufficiently far in a precise sense).

When $S$ is the union of two convex sets $S_1,S_2$, the situation
may become more complicated, even when both sets
are calibrable. In particular, if the sets are near each other,
more precisely if the perimeter of its convex hull is smaller
than the sum of the perimeters of the two sets,
then the sets interact, and the solution outside $S$ is not null (for small values of $\lambda$). This is the case, for instance, when
$S$ is the union of two balls in $\R^2$, which is the object of this paper.

Let us mention in this context the work of
Allard \cite{All09} who calculated the solution of \eqref{vpI} when $S$ is the union of two balls with the same radius.
He also computed the solution of \eqref{vpI} when $S$ is the union of two squares.
In this paper we extend the result of Allard \cite{All09} to the case where $S$ is the union of any two balls in $\R^2$.
The interesting case is when the perimeter of the convex hull of $S$ is less than the perimeter of $S$, since otherwise the solution can be described as the sum of the solutions corresponding to each ball \cite{BeCaNo:00}. Our approach differs from the one in \cite{All09} even
for the case of two balls of the same radius. While the solution in \cite{All09} is obtained by a explicit computation, we
describe it in a shorter way by means of more general geometric arguments. Our starting point is the observation that
$u_\lambda$ is a solution of~\eqref{vpI} if and only if the sets $[u_\lambda \geq s]$ minimize the variational problem \cite{CCN-MMS,All09}
\begin{equation}\label{vpsI}
 \Fsl(X):=P(X)
+ \frac{s}{\lambda}\abs{ X \setminus S}
- \frac{(1-s)}{\lambda} \abs{ X \cap S}
\qquad s\in [0,1], \lambda>0\;,
\end{equation}
where $P(X)$ is the perimeter of $X$ (and we understand that $P(X) =+\infty$ if $\chi_X \not \in BV(\R^2)$).
Let us point out that, for $\lambda > 0$ fixed, the solutions of \eqref{vpsI} are monotonically decreasing as $s$ increases,
and can be then packed together to build up a function which solves \eqref{vpI} \cite{Alter.Caselles.Chambolle05,CCN-MMS}.
Thus, to compute the solution of \eqref{vpI} we study the solution of \eqref{vpsI}
and those solutions can be constructed by means of geometric arguments.

On one hand, the Euler-Lagrange equation of \eqref{vpsI} tells us that, if $\Cs$ is a minimizer of \eqref{vpsI}, then
$\partial \Cs$ is $\mathcal{C}^{1,1}$, $\partial \Cs$ has
curvature $\frac{1-s}{\lambda}$ inside $S$ and $-\frac{s}{\lambda}$ outside $S$.
When $S$ is the union of two disjoint open balls $S_1,S_2$ and $\lambda \leq
r_c$, for some value of $r_c  > 0$ that depends on the geometry of $S$, 
we prove that the intersection of $S_i,i=1,2$, with the minimizing sets is either $S_i$ or $\emptyset$.
We also give a counterexample showing that this result is not true for any value of $\lambda$.
Thus, for $\lambda \leq r_c$, the possible minimizers of \eqref{vpsI} are:
$\emptyset, S_1,S_2, S, \close{\raussen}{S}$, where $\close{r}{S}$
denotes the $r$-closing of the set $S$, that is, the complement of the union of the balls of radius $r$ contained in $\R^2\setminus S$.

The computation of explicit examples of TV denoising permits to exhibit qualitative features of the solution. In particular, the appearance of new level lines is a undesirable feature for denoising. 
Better denoising algorithms have been developed in the last few years \cite{buades2006review}.

Let us describe the plan of the paper. 
In Section  \ref{sect:preliminaries} we
review some known results that permit to set the context of our analysis. 

In Section \ref{sect:solutionballs}
we describe the generic properties of the minimizers $C_{s,\lambda}$ of \eqref{vpsI} and we prove that, if $S$ is the union of two balls
and $\lambda \leq r_c$, then
the intersection of $C_{s,\lambda}$ with $S_i$, $i=1,2$, is either $S_i$ or the
empty set. 

\changed{
This permits to
reduce the set of possible minimizers of \eqref{vps} to the following six ones: $\emptyset, S_1,S_2, S, \close{\raussen}{S}, \Gamma_{s,\lambda}(S)$.
In Section \ref{sect:construction} we explain how to construct the sets $\Gamma_{s,\lambda}(S)$ as well as how to construct the explicit solutions for 
\eqref{vpI}. The proof of the main theorem can be found in Section \ref{sect:proof}.
In Section \ref{sect:dualnorm} we describe how to calculate the dual norm of the function $\chi_S$. 
}

\smallskip

\noindent {\bf Acknowledgements.} 
M. Novaga was partially supported by the Italian INDAM-GNAMPA and by the University of Pisa via grant PRA-2015-0017.
The work of C. P\"oschl was supported by the Austrian Science Fund
(FWF): Projects J-2970 (Schr\"odinger scholarship), T644-N26
(Hertha-Firnberg-fellowship).

This paper was inspired by our coauthor and friend Vicent Caselles.
His passion and his strong motivation were a continuous stimulus in our research, and we dedicate this work to his memory.

\section{Preliminaries}\label{sect:preliminaries}

\subsection{Total variation and perimeter}

Let $\Omega$ be an open subset of $\R^2$. A function $u\in
L^1(\Omega)$ whose gradient $Du$ in the sense of distributions is
a (vector valued) Radon measure with finite total variation in
$\Omega$ is called a function of bounded variation. The class of
such functions will be denoted by $BV(\Omega)$. The total
variation of $Du$ on $\Omega$ turns out to be
\begin{equation*}
\sup \left\{\int_\Omega u~ \Div z~dx : z \in
C^\infty_0(\Omega; \R^2), \Vert z \Vert_{L^\infty(\Omega)} := {\rm
ess}\sup_{\!\!\!\!\!x \in \Omega} \vert z(x)\vert \leq 1\right\},
\end{equation*}
and will be denoted by $\vert
Du\vert (\Omega)$ or by $\int_\Omega \vert Du\vert$. $BV(\Omega)$ is a Banach
space when endowed with the norm $\int_\Omega \vert u \vert ~dx +
\vert Du\vert (\Omega)$.

Let us denote by $\H^{1}$ the one-dimensional
Hausdorff measure.

A measurable set $E \subseteq \R^2$ is said to have finite
perimeter if $\chi_{E}\in BV(\R^2)$. 
The perimeter of $E$ is defined as $P(E):= \vert D
\chi_{E} \vert (\R^2)$. We recall that when $E$ is a finite-perimeter set
with regular boundary (for instance, Lipschitz),
its perimeter $P(E)$ also coincides with the more standard definition
$\H^{1}(\partial E)$.
For more properties and references on functions
of bounded variation we refer to \cite{AmbFusPal00}. We also mention
the following review papers on applications to image analysis and denoising.
\cite{buades2006review,ChaCasCre10,CasChaNov11}.

\subsection{Morphological operators}

\begin{definition}[Opening and Closing operators]
For any set $X$ and $r>0$, let $B_r(x)$ be a ball with radius $r$ and center
$x$. 
We define the \textit{opening} and  the \textit{closing}  
of $X$, with radius $r$, respectively  by
\begin{eqnarray*}
 \open{r}{X}&:=&\bigcup_{x: B_r(x) \subset X} B_r(x)\;,
 \\
\close{r}{X}&:=& \lr{\open{r}{X^C}}^C\;,
\end{eqnarray*}
where $X^C$ denotes the complement of the set
$X$.
\end{definition}

The opening operator is anti-extensive ($\open{r}{X} \subset X$), 
conserves the subset property ($X\subset Y$ then $\open{r}{X} \subset
\open{r}{Y}$) and
is idempotent ($\open{r}{X} = \open{r}{ \open{r}{X}}$).
For more on application of morphological operators we refer to \cite{Soi03}.
Later we need the following properties of the opening and closing operator.

\begin{lemma}[Properties of the Opening and Closing operator]
\label{le:propertiesOpeningClosing}
Let $S$ be an arbitrary set. 
The curvature of $\partial\close{r}{X}$ is larger or equal to $-\frac{1}{r}$,
the curvature of $\partial\open{r}{X}$ is less or equal to $\frac{1}{r}$.
Consequently the curvatures of $\partial \close{r}{X} \setminus X, \partial
\open{r}{S} \cap S$ are 
$-\frac{1}{r}, \frac{1}{r}$ respectively.
Moreover, if $\close{r}{X} \not =X$, then $\min\set{\kappa(\partial
X)}<\frac{1}{r}$. 
If $\open{r}{X} \not = X$, then $\max\set{ \kappa(\partial X)}>\frac{1}{r}$.
\end{lemma}

\begin{proof}
Assume that there exists a point $A \in \open{r}{X}$, with curvature
$>\frac{1}{r}$, 
then there exists no circle touching $\open{r}{X}$ at $A$ that lies 
inside $\open{r}{X}$, this contradicts the definition of the opening operator, 
hence we can conclude that the curvature of $\partial \open{r}{X}$ is smaller or
equal to $\frac{1}{r}$.
The proof of the second estimate is analog.
\end{proof}

\subsection{Review of some basic results}\label{sect:R}

The following result was proved in \cite{All09,CCN-MMS}.

\begin{proposition}\label{cc}
Let $S  \subset \R^2$ be a bounded measurable set. Then there is a unique solution $u_\lambda$ of
(\ref{vpI}), which satisfies the Euler-Lagrange equation
\begin{equation}\label{euler}
u_\lambda - \lambda\; \mathrm{div}\, z = \chi_S \qquad \hbox{\rm in $\R^2$},
\end{equation}
where $z:\R^2\to\R^2$ is such that $\Vert z\Vert_\infty\leq 1$
and $z\cdot Du_\lambda = |Du_\lambda|$.

Moreover, for any $s\in \R$,
$\{u_\lambda \geq s\}$ (resp. $\{u_\lambda > s\}$) is the maximal (resp. the minimal) solution of
\begin{equation}\label{vps}
\min_{X \subset \R^2} \Fsl(X) := P(X) + \frac{s}{\lambda} \vert X
\setminus S\vert - \frac{(1-s)}{\lambda} \vert X \cap S\vert.
\end{equation}
In particular, for all $t$ but a countable set the solution of (\ref{vps}) is unique.

Conversely, for any $s\in\R$, let $Q_s$ be a solution of (\ref{vps}). If $s > s'$, then
$Q_s \subseteq Q_{s'}$. The function
$$
u(x) = \sup \{s: x \in Q_s\}
$$
is the solution of (\ref{vpI}).
\end{proposition}

Thus, in order to build up the solution of (\ref{vpI}) it suffices to compute the solutions
of the family of problems (\ref{vps}). This will be the strategy we follow to compute the explicit solution
when $S$ is the union of two balls.

Recall that if $g \in L^2(\R^2)$ the dual $BV$-norm of $g$ is given by
$$
\Vert g\Vert_* = \sup_{u\in BV(\R^2), \vert Du\vert (\R^2)\leq 1}
\int_{\R^2} g u \, dx.
$$
Then $\Vert g\Vert_* \leq 1$ if and only if $$\int_{\R^2} g u \,
dx \leq \int_{\R^2} \vert Du\vert$$ for any $u\in BV(\R^2)$. This
is equivalent to say that
$$\int_{F} g \leq P(F), \qquad \hbox{for any set $F$ of finite perimeter.}$$

Let us first recall a result that permits to compute the value of $\lambda$ for which
the solution $u_\lambda=0$. The result was proved in \cite{Meyer,BeCaNo:00}.

\begin{proposition}\label{p2} (\cite{BeCaNo:00})
Let $g\in L^2(\R^2)$. Let us consider the problem:
\begin{equation}\label{vp1}
\min_{u \in BV(\R^2)\cap L^2(\R^2)}\, \F_\lambda(g)
\end{equation}
where
\begin{equation}\label{eqfla}
\F_\lambda(g):= \int_{\R^2} 
\vert Du\vert + \frac{1}{2\lambda} \int_{\R^2} \vert u - g\vert^2\, dx.
\end{equation}
The following conditions are equivalent
\begin{itemize}
\item[{\rm (i)}] $u=0$ is a solution of \eqref{vp1}.
\item[{\rm (ii)}]  $\Vert
g\Vert_* \leq \lambda$
\item[{\rm (iii)}]  There is  a
vector field $\xi\in L^\infty(\R^2,\R^2)$, $\Vert
\xi\Vert_\infty\leq 1$ such that $-\mathrm{div} \,\xi = g$.
\end{itemize}
\end{proposition}

The following result has been proved in \cite{Giusti:78,BeCaNo:00,Alter.Caselles.Chambolle05b,KLR}.

\begin{theorem}\label{evol_shapes} 
Let $C \subset \R^2$ be a bounded set of finite perimeter, and
assume that $C$ is connected. Let $\gamma > 0$. The following
conditions are equivalent:
\begin{itemize}
\item[(i)] $C$ decreases at speed $\gamma$, i.e, for any $\lambda > 0$ $u_\lambda :=
\left(1 - \lambda \gamma \right)^+ \chi_C(x)$ is the solution of
(\ref{vpI}) corresponding to $\chi_C(x)$.
\item[(ii)] $C$ is convex, $\gamma = \gamma_C := \frac{P(C)}{|C|}$ and minimizes the
functional $$ {\cal G}_{\gamma_C}(X) := P(X) - \gamma_C \vert
X\vert, \qquad X \subseteq C, \ X ~{of~ finite~ perimeter}. $$
\item[(iii)] $C$ is convex, $\partial C$ is of class $C^{1,1}$,
$\gamma = \gamma_C := \frac{P(C)}{|C|}$, and the following inequality holds:

\begin{equation*}
{\rm ess} \displaystyle \sup_{\!\!\!\!\!\!p\in
\partial C} \kappa_{\partial C}(p) \leq \gamma_C,
\end{equation*}
where $\kappa_{\partial C}(p)$ denotes the curvature of $\partial
C$ at the point $p$.
\end{itemize}
\end{theorem}

For all $r\in\R$, we set $r^+ := \max\set{0,r}$.
The following result has been proved in \cite[Theorem 7 and Proposition 8]{BeCaNo:00}.

\begin{lemma}\label{destruccio}
Let $S_1,S_2\subset  \R^2$ be two disjoint balls, let $S=S_1\cup S_2$ and $f=\chi_{S}$.
Then  $$u_\lambda=
\left(1-\frac{P(S_1)}{|S_1|}\lambda \right)^+ \chi^{}_{S_1} +
\left(1-\frac{P(S_2)}{|S_2|}\lambda \right)^+ \chi^{}_{S_2}$$
is a solution of (\ref{vpI}) for any $\lambda > 0$ if and only if
\begin{equation}\label{lastRR}
P(S)\leq P(\mathrm{co}(S))\;,
\end{equation}
where $\mathrm{co}(S)$ denotes the convex envelope of $S$.
In other words, the solution of (\ref{vpI})
is the sum of the two solutions corresponding to $\chi_{S_1}$ and $\chi_{S_2}$
if and only if (\ref{lastRR}) holds.
\end{lemma}
In the general case the minimizers of $\Fsl$ can be subsets of $S$ or contain parts outside $S$, as we shall see in the following section.

\section{Properties of minimizers}\label{sect:solutionballs}
As explained in Section \ref{sect:R} our purpose is to characterize the minimizers of $\Fsl$
when $S$ is the union of two balls $S_1,S_2$ with disjoint interiors and distance $d$.  In order to fix the notation we assume
$S_1,S_2$ are open balls of radii $r_1\ge r_2$.



Let us first state a simple geometric result which will be useful in the proof of Proposition \ref{lemma:gg} below.

\begin{lemma}\label{threeballs}
Let $B_1$, $B_2$, $B_3$ be three open balls of equal radius,
intersecting $\partial S_2$ at equal angles.
Let $\Gamma_{S_2}$ be the arc of $\partial S_2$ contained in $\mathrm{co}(S)$.
Assume the three balls intersect $\Gamma_{S_2}$ and $B_3$ is between $B_1$ and $B_2$ when we go along $\Gamma_{S_2}$
(see Figure \ref{fig:threeballs}).
If $S_1$ intersects $B_1$ and $B_2$, then it intersects also $B_3$. The same statement holds interchanging $S_1$ and $S_2$.
\end{lemma}

\begin{figure}
\begin{center}
\includegraphics{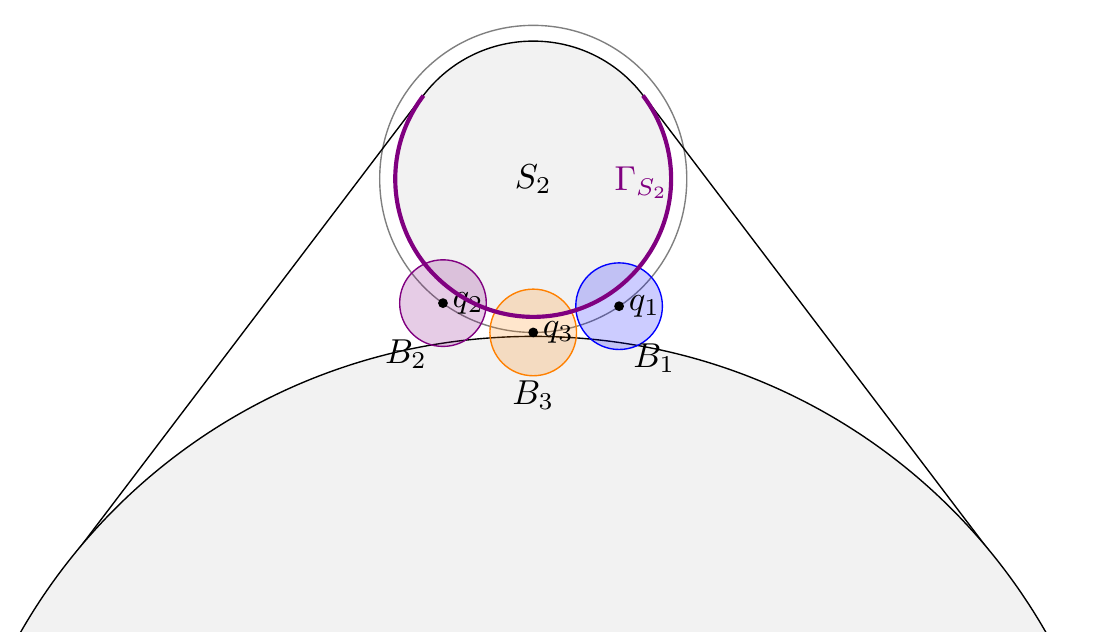}
\end{center}
\caption{The construction in the proof of Lemma \ref{threeballs}.}\label{fig:threeballs}
\end{figure}

\begin{proof}
Observe that the centers of $B_1$, $B_2$, and $B_3$, denoted respectively by $q_1,q_2,q_3$,
are contained in a circle concentric with $S_2$.
Let $p$ be the center of $S_1$ and $r$ be the  common radius of $B_i$,  $i=1,2,3$.
Consider the triangle formed by the segments $[p,q_1]$, $[p,q_2]$ and $[q_1,q_2]$. Notice that
since $S_1$ intersects $B_1$ and $B_2$, $|p-q_1| \le r_1 + r$ and $|p-q_2| \le r_1+ r$. Since $q_3$ is contained in
the interior of such triangle then $|p-q_3| < r_1+ r$,
and therefore $S_1$ intersects $B_3$.
\end{proof}

\begin{proposition}\label{lemma:gg}
Let $\Cs$ be a minimizer of $\Fsl$. 
Then the boundary $\partial \Cs$ is of class $C^{1,1}$,
$\Cs\subset \overline{{\rm co}(S)}$,
and one of the following possibilities holds:
\begin{enumerate}[a)]
\item \label{it:gg1}
$\Cs\in\{\emptyset,S_1,S,\close{\frac\lambda s}{S}\}$, 
and $\Cs\ne S_2$ if $r_1>r_2$;
\item \label{it:gg2}
$S_1\subset \Cs$, $\partial \Cs\cap S_2$ is a circular arc 
with curvature $\frac{1-s}{\lambda}$,
and $\partial \Cs \setminus \overline{S}$ is composed by two circular arcs
with curvature $-\frac{s}{\lambda}$.
\end{enumerate} 
\end{proposition}

\begin{proof}
The regularity of $\partial \Cs$ is a classical result \cite{Ambrosio}.
The Euler-Lagrange equations say that, if non-empty, $\partial \Cs \setminus \overline{S}$ are arcs of circle of curvature
$-\frac{s}{\lambda}$, and $\partial \Cs \cap  S$ are arcs of circle of curvature  $\frac{1-s}{\lambda}$.
In particular, $\partial \Cs$ has finitely many connected components which are $C^{1,1}$ Jordan curves,
and any two of them have positive distance. 

Notice that the energy is additive on the connected components, that is, if $\mathcal{CC}_{s,\lambda}$ denotes the set of connected components of
$\Cs$, then $\Fsl(\Cs) =\sum_{C\in\mathcal{CC}_{s,\lambda}} \Fsl(C)$. Moreover $\Fsl(C) \leq 0$ for any $C\in\mathcal{CC}_{s,\lambda}$, otherwise
we can eliminate this component decreasing the energy.
Let $C$ be a connected component of $\Cs$. Modulo null sets, if $C \cap S_i=\emptyset$, $i\in \{1,2\}$,
then $C\subseteq S_j$, $j\in\{1,2\}$, $j\neq i$. Otherwise, by replacing $C$ by $C\cap S_j$  we decrease the energy
of $\Cs$. Thus there are only three possibilities: $C\subseteq S_1$, $C\subseteq S_2$, or $C \cap S_1\neq \emptyset$
and $C \cap S_2\neq \emptyset$.


Without loss of generality, we can assume that $\mathrm{dist}(S_1,S_2)  > 0$.
Having proved the result in this case, by passing to the limit we get it also
when $\mathrm{dist}(S_1,S_2) = 0$. We divide the rest of the proof in several steps.
Without loss of generality we may assume that $\Cs$ is an open set.

\medskip\noindent {\it Step 1.} \textit{If $r_1>r_2$ then $\Cs\ne S_2$}. 

Assume by contradiction that $\Cs=S_2$, 
then $\Fsl(\Cs)= \Fsl(S_2)= 2\pi r_2 - \frac{(1-s)\pi}{\lambda} r_2^2\le 0$,
which implies $s<1$ and $r_2\ge \frac{2\lambda}{1-s}$.
This in turn implies that $\Fsl(S_1)<\Fsl(S_2)$, 
contradicting the minimality of $\Cs$.

\medskip\noindent {\it Step 2.} \textit{We have $\Cs\subset {\rm co}(S)$}.

Being ${\rm co}(S)$ convex, this follows from the fact that
$\Cs\cap {\rm co}(S)$ has lower energy than $\Cs$, with equality if and only if
$\Cs\subset {\rm co}(S)$.

\medskip\noindent {\it Step 3.}
\textit{Let $C$ be a connected component of $\Cs$ intersecting only one of the two circles, say $S_i$, then $C=S_i$}.

Replacing $C$ with $C\cap S_i$ decreases the energy, hence we may assume $C\subset S_i$.
On the other hand, $C$ does not have holes since by filling them we would also decrease the energy.
Since $\partial C$ is $C^{1,1}$, then $C$ is a ball of radius $r(s)=\frac{\lambda}{1-s}$. As we observed at
the beginning of the proof, it is at positive distance from the other connected components. Thus we may dilate it
to a ball $B_r$ of radius $r$ contained in $S_i$. Since
$\Fsl(B_r) = 2\pi r - \frac{1-s}{\lambda} \pi r^2$, for $r>r(s)$ near $r(s)$  we have
$\Fsl(B_r) < \Fsl(C)$ and this permits to decreases the energy of $\Cs$. Thus $C=S_i$.

\medskip\noindent {\it Step 4.} \textit{Let $\Gamma$ be a connected component of $\partial\Cs$, 
and assume that $\Gamma\setminus \overline S$ is nonempty. 
Then $\Gamma\setminus \overline S$ consists of arcs joining $S_1$ and $S_2$}.  

Assume by contradiction that $\Gamma$ contains an arc with both extrema on $\partial S_i$.
Without loss of generality we can assume $i=1$.
Then $\Gamma\setminus \overline S_2$
is the union of consecutive arcs which are alternatively in
$S_1$ and in $\R^2\setminus \overline{S}$. 
By Lemma \ref{threeballs}, all the arcs of $\Gamma\setminus \overline{S}$ except the two extremal ones are similar,
that is, they coincide after a rotation around the center of $S_1$ (see Figure 
\ref{fig:wandering2}).
In particular, at least one of these arcs intersects the complementary of 
${\rm co}(S)$, contradicting {\it Step 2.}

\medskip\noindent {\it Step 5.} \textit{Let $\Gamma$ be a connected component of $\partial\Cs$ that intersects both $S_1$ and $S_2$.
If $\Gamma \cap S_i \neq \emptyset$, for $i=1,2$,
then $\Gamma\cap  S_i$ is an arc of circle of radius
$\frac{\lambda}{1-s}$ and $\Gamma\setminus \overline{S}$ consists of two arcs of circle
of radius $\frac{\lambda}{s}$, connecting $S_1$ and $S_2$}.

Let $\ell$ be the line passing through the centers of $S_1$ and $S_2$.
Let us consider a coordinate system where the $y$-axis coincides with $\ell$, 
and $S_2$ is above $S_1$. Let $\Gamma_{S_2}$ be the arc of $\partial S_2$ contained in $\mathrm{co}(S)$. By going along
$\partial S_2$ in the counterclockwise direction we induce an order in $\Gamma_{S_2}$. 
Similarly, if $\Gamma_{S_1}$ denotes the arc of $\partial S_1$ contained in
$\mathrm{co}(S)$, we consider the order in
$\Gamma_{S_1}$  induced  by going  along $\partial S_1$  clockwise. 

Let us order $\Gamma$ counterclockwise.
Since $\Gamma\setminus \overline{S}\neq \emptyset$,
we may choose $G$ as the arc in $\Gamma\setminus S$ having greatest intersection point with $\Gamma_{S_2}$,
with respect to the order of $\Gamma_{S_2}$. The arc $G$ intersects $\Gamma_{S_1}$ at  point $q$, and $\Gamma_{S_2}$ at a point $p$.
Let $\gamma_{S_2}$ be the arc of $\Gamma\cap S_2$ starting at $p$ (see Figure \ref{fig:wandering2}
). Let us observe that if
$p_1$ is the other endpoint of $\gamma_{S_2}$, then $p_1 \in \Gamma_{S_2}$ and $p_1 < p$ with respect to the order of 
$\Gamma_{S_2}$. Thus $G$ continues after $\gamma_{S_2}$ with an arc $G_1\subset \Gamma\setminus S$ until
it intersects $S_1$ at a point $q_1$ (see Figure \ref{fig:wandering2}, left).
Then $G_1$ enters into $S_1$ at a point $q_1 < q$.
As we observed above, 
there is an arc $\gamma_{S_1}\subset \Gamma\cap S_1$ that starts at $q_1$ and exits
from $S_1$ at $q_2$.

Let $G_2$ be the arc in $\Gamma\setminus S$ that starts at $q_2$.
We claim that $G_2 = G$. Indeed, notice first that $q_2 \le q$ by the choice of $G$.
On the other hand,
if $q_2<q$ we could continue following $\Gamma$ along arcs of circles inside and outside $S$, until we would reach some point where 
these arcs intersect each other, giving a contradiction.
We thus conclude that $q = q_2$ and hence $G_2 = G$.

\begin{figure}
\begin{center}
\includegraphics[scale=0.8]{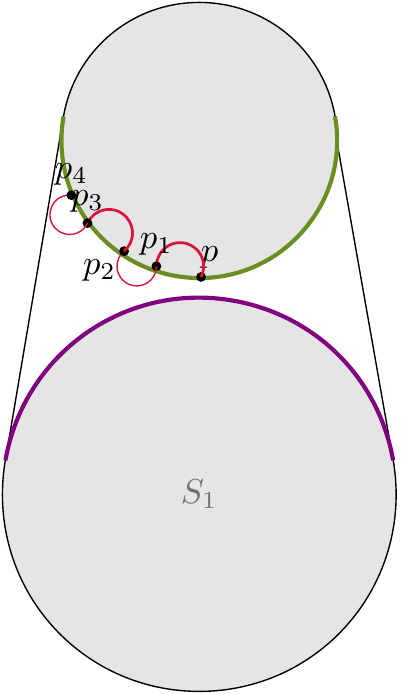}
\hspace{4cm}
\includegraphics{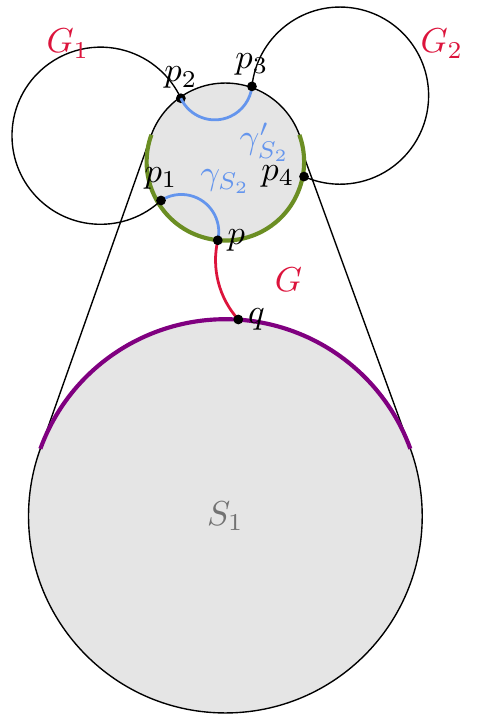}
\end{center}
\caption{
Left: The construction in {\it Step 4}.
Right: Illustration of $G,G_1$, $G_2$, $\gamma_{S_2}$, $\gamma_{S_2}'$, 
and the points $q,p,p_1,p_2$ in {\it Step 5}. 
}
\label{fig:wandering2}
\end{figure}

\medskip\noindent {\it Step 6.}
\textit{Let $\Gamma$ be a connected component of $\partial \Cs$.
Then $\Gamma\setminus \partial S$ is a union circular arcs with angular span strictly less than $\pi$.}

Let $K$ be a connected component of $\Gamma\setminus\partial S$. Then $K$ is a circular arc 
of radius $r=r(s,\lambda)$, where $r=\lambda/(1-s)$ if $K\subset S$, and $r=\lambda/s$ 
if $K\subset \R^2\setminus\overline S$ (if $s=0$ then $K$ is a segment).

Assume by contradiction that the angular span $\alpha$ of $K$ is greater or equal to $\pi$.
Then we can modify $\Cs$ and construct a new set with lower energy.
Indeed, for $\epsilon > 0$ small enough, we consider a ball $B_\epsilon$ of radius $r_\epsilon = (1+\epsilon)r$,
containing the endpoints of $K$.
Let $K_\epsilon\subset\partial B_\epsilon$ be the circular arc with the same endpoint as $K$,
and let $C_\epsilon$ be the such that $\partial C_\epsilon = (\partial \Cs \setminus K)\cup K_\epsilon$.
It is easy to check that $\Fsl(C_\epsilon) < \Fsl(\Cs)$, contradicting the minimality of $\Cs$.

\medskip\noindent {\it Step 7.} \textit{Let $C$ be a connected component of $\Cs$,
then $C$ is simply connected.} 

If $C$ intersects only $S_i$ then $C=S_i$ be {\it Step 3}, hence we can assume that $C$
intersects both $S_1$ and $S_2$. If $C$ is not simply connected then $\partial C$
contains a closed Jordan curve $\Gamma$ which bounds a bounded connected component of $\R^2\setminus C$.
By the previous discussion we can write $\Gamma=\cup_{i=1}^4\Gamma_i$, where 
$\Gamma_i$ are circular arcs, $\Gamma_1,\Gamma_2$
have curvature $-(1-s)/\lambda$ and are contained in $S_1,S_2$ respectively, and
$\Gamma_3,\Gamma_4$ have curvature $s/\lambda$ and are contained in $\R^2\setminus S$.

Since the curvature $\kappa$ of $\Gamma$ is negative on $\Gamma_1\cup\Gamma_2$
and positive on $\Gamma_3\cup\Gamma_4$,  we have
$$
\int_{\Gamma_3\cup\Gamma_4} \kappa\, d\mathcal{H}^1 = 2\pi
- \int_{\Gamma_1\cup\Gamma_2} \kappa\, d\mathcal{H}^1 \geq 2\pi.
$$
On the other hand,
$$
\int_{\Gamma_3\cup\Gamma_4} \kappa\, d\mathcal{H}^1 < 2\pi
$$
since by {\it Step 6} we know that
$\Gamma_i$ have all angular span strictly less than $\pi$.

\medskip\noindent {\it Step 8.} 
\textit{Let $C$ be a connected component of $\Cs$ intersecting both $S_1$ and $S_2$, 
then $C$ contains $S_1$ or $S_2$. In particular, the set $\Cs$ is connected.}

If $C$ contains neither $S_1$ nor $S_2$, 
we can write $\partial C=\cup_{i=1}^4\Gamma_i$, where 
$\Gamma_i$ are circular arcs, $\Gamma_1,\Gamma_2$
have curvature $(1-s)/\lambda$ and are contained in $S_1,S_2$ respectively, and
$\Gamma_3,\Gamma_4$ have curvature $-s/\lambda$ and are contained in $\R^2\setminus S$.
Reasoning as in {\it Step 7} we then reach a contradiction. 

Assume now that $\Cs$ is not connected and let $\tilde C$ be a connected component different from $C$. 
By the previous discussion, $\tilde C$ contains either $S_1$ or $S_2$, hence it intersects $C$,
thus giving a contradiction.

\medskip\noindent {\it Step 9.} \textit{$\Cs$ is symmetric with respect to $\ell$. Moreover, if $r_1=r_2$, then $\Cs$ is also symmetric with respect to the 
line $\ell'$ which is orthogonal to $\ell$ and has the same distance from $S_1$ and $S_2$.}

Let $\tCs$ be the set obtained by reflecting $\Cs$ through $\ell$,
which is still a minimizer of $\Fsl$.
Letting $A=\Cs\cap \tCs$, $B=\Cs\cup\tCs$,
we have
$$
\Fsl(A)+\Fsl(B)=\Fsl(\Cs)+\Fsl(\tCs),
$$
which implies that both $A$ and $B$ are minimizers of $\Fsl$. 
In particular, $A$ and $B$ have boundaries of class $C^{1,1}$,
and this is possible only if $\Cs=\tCs$.

The second assertion can be proved analogously by replacing $\ell$
with $\ell'$ in the reflection argument.

\medskip\noindent {\it Step 10.} \textit{If $\Cs$ is nonempty and different from
$S_2$ then it contains $S_1$. If $r_1=r_2$ then $\Cs$ contains $S$.}

Assume by contradiction that $\Cs$ does not contain $S_1$. Then
from the previous steps it follows that $\Cs$ contains $S_2$ 
and intersects $S_1$ in a circular arc. If $r_1=r_2$ this violates the symmetry 
of $\Cs$ with respect to $\ell'$ and gives a contradiction.

Let us consider the case $r_1>r_2$, and
let $\tCs$ 
(resp. $\tilde S_1$)
be the sets obtained by reflecting $\Cs$ 
(resp. $S_1$) through $\ell'$. 
Let also 
\[
A = \Cs\cap\left(\tilde S_1\setminus S_2\right) 
\qquad 
B = \tCs\cap\left(\tilde S_1\setminus S_2\right) 
\]

\begin{figure} \begin{center}
\includegraphics[]{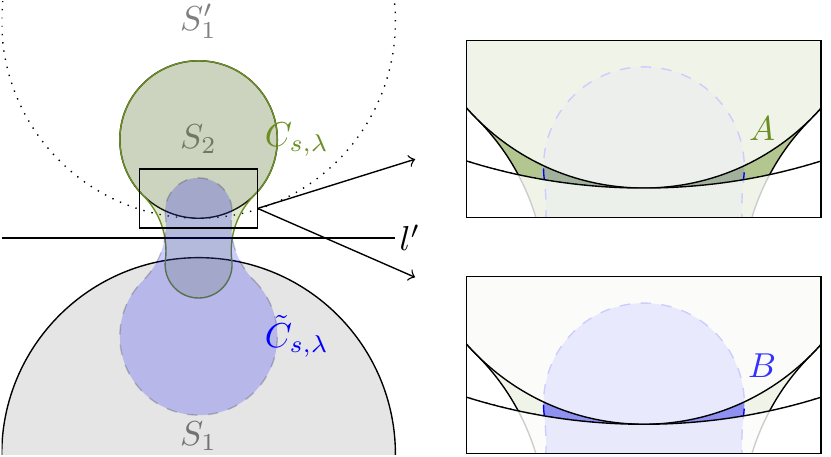}
 \end{center}
\caption{Illustration of {\it Step 10}.}
\end{figure}

It is easy to check that $B\subset A$ and 
\[
\Fsl(\Cs)-\Fsl(\tCs) = \frac 1\lambda\left(|A|-|B|\right)> 0\,, 
\]
contradicting the minimality of $\Cs$.

\medskip\noindent {\it Step 11.} From the previous discussion it follows that
either $\Cs \in\{\emptyset, S_1,S\}$, or $\Cs$ is simply connected, contains $S_1$
and intersects $S_2$.

\end{proof}

\begin{corollary}[\cite{All09}]\label{lemma:ffequalradius}
Assume that $S_1$ and $S_2$ are two open disjoint balls with equal radius.
Let $\Cs$ be a minimizer of $\Fsl$, $s \in [0,1]$.
Then if the set $\Cs$ is non-empty, the boundary $\partial \Cs$ is of class $C^{1,1}$.
Moreover, $\Cs\cap S_i = S_i$ or $\emptyset$ for any $\lambda >0$, $s\in [0,1]$ and $i=1,2$.
In particular, when they exist, the arcs of $\partial \Cs \setminus \overline{S}$ have radius
$\frac{\lambda}{s}$ and are tangent to $\partial S$.
\end{corollary}

\noindent\textbf{Example.} We give an example of 
a situation where case 
\ref{it:gg2} of Proposition \ref{lemma:gg} is realized.
For that, we consider two disjoint balls $S_1$ and $S_2$ and assume that they are tangent. Assume also that
$r_2 < 1 < r_1$ and take $\lambda = 1$. Then for an appropriate choice of $r_1,r_2$, the function
has level sets that are transversal to $S_2$, that is, they intersect $S_2$ but do not contain it. 
Let $C_\lambda = \{u_\lambda > 0\}$. Since $C_\lambda$ is a minimizer
of the functional $\mathcal{F}_{0,1}(X)=P(X) - \vert X\cap S \vert$, 
it follows that the maximum of the curvature
of $\partial C_\lambda$ is less than $1$. However, 
if $C_\lambda \supseteq S_2$, then the maximum of the curvature 
is $\frac{1}{r_2}>1$, is less than $1$, contradicting our choice of
$r_2$.
If we prove that $C_\lambda \neq S_1$, then $C_\lambda$ is of the type described in Proposition \ref{lemma:gg} \ref{it:gg2}. 
For that,  it suffices to show that $\mathcal{F}_{0,1}(\mathrm{co}(S)) -
\mathcal{F}_{0,1}(S_1) < 0$. Indeed, for $r_2 \ll r_1$, this difference is bounded by
$$
\eta = C \frac{r_2^{3/2}}{r_1^{1/2}} - \pi r_2^2,
$$
for some constant $C > 0$ independent of $r_1,r_2$. 
If we choose $r_1 = \frac{M}{r_2}$ and $M > \frac{C^2}{\pi^2}$, then $\eta = \left(\frac{C}{M^{1/2}}-\pi\right) r_2^2 < 0$.

\begin{definition}\label{deftrans}
\changed{ We call  {\rm transversal sets} the sets satisfying 
condition \ref{it:gg2} in Proposition \ref{lemma:gg}, and we denote them by $T_{s,\lambda}$.
If for a given couple of $(s,\lambda)$ we have two transversal sets, denoted by
$T^+_{s,\lambda}, T^-_{s,\lambda}$, with  
$T^+_{s,\lambda}\subset T^-{s,\lambda}$, then we say, that
 $T^+_{s,\lambda}$ is of increasing type and $T^-_{s,\lambda}$ is of decreasing type
 (see Figure \ref{transversal}).
 }
\end{definition}
\changed{
Assume that for both combinations 
$
(s_1,\lambda_1), 
(s_2,\lambda_2)$ with according radii 
$t_1=\frac{\lambda_1}{1-s_1}< 
  t_2=\frac{\lambda_2}{1-s_2}$ we have two transversal sets 
that we denote by 
$\set{T_{s_1,\lambda_1}^-, T_{s_1,\lambda_1}^+ }, 
  \set{T_{s_2,\lambda_2}^-, T_{s_2,\lambda_2}^+}$ respectively. 
Moreover, by definition of increasing and decreasing transversal sets, we have
$T_i^- \subset T_i^+, i=1,2$. 
Then $T_1^- \subset T_2^-$ and 
$T_2^+ \subset T_1^+$, 
meaning that if we increase the radius from $t_1$ to $t_2$ the sets of increasing type increase and the sets of decreasing type decrease.
}

\begin{figure}
\begin{tikzpicture}[fill=black!5!,inner sep=1pt, 
place/.style={circle,draw=black,draw opacity=0,fill=black}]
\filldraw (0,0) circle (1);
\filldraw (1.6,0) circle (0.5);
\def\alpha{71}
\def\r{0.53}
\draw[blue]  (\alpha:1) -- +(\alpha-90 :1.39) node[] (x1) {} arc(\alpha:-\alpha:\r) -- (-\alpha:1);
\def\beta{69}
\draw[red]  (\beta:1) -- +(\beta-90 :1.25)   arc(\beta:-\beta:\r)node[]  (x2) {} -- (-\beta:1);
\draw[blue,<-] (x1) -- +(1,0.5) node[right] {outer transversal set - decreasing type};
\draw[red,<-] (x2) -- +(1.1,-0.2) node[right] {inner transversal set - increasing type};
\end{tikzpicture}
 \begin{tikzpicture}[fill=black!5!,inner sep=1pt, 
place/.style={circle,draw=black,draw opacity=0,fill=black}]
\filldraw (0,0) circle (1);
\filldraw (1.6,0) circle (0.5);
\def\alpha{60}
\def\r{0.53}
\def\rr{0.4}
\draw[red!90!]  (\alpha:1) -- +(\alpha-90 :1.0) node[] (x1) {} arc(\alpha:\alpha+360:\rr) arc(\alpha:-\alpha:\rr) -- (-\alpha:1);
\def\beta{69}
\draw[red!50!]  (\beta:1) -- +(\beta-90 :1.25)   arc(\beta:360+\beta:\r) arc(\beta:-\beta:\r)node[]  (x2) {} -- (-\beta:1);
\draw[blue,<-] (x1) -- +(1,0.5) node[right] {increasing type};
\end{tikzpicture}
\caption{Left: for some values of $(s,\lambda)$, there are two transversal sets.
We call the inner one of increasing type, because for increasing radius (connected to $s,\lambda$, these sets decrease, whereas the decreasing sets decrease. Right:
For transversal sets of increasing type, by increasing the radius of the inner arc, the transversal set increases.}
\label{transversal}
\end{figure}
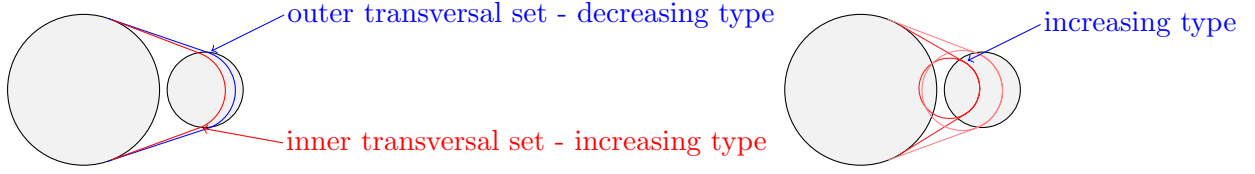

The following Lemma is needed later to state that transversal sets of increasing type cannot be minimizers of $\Fsl$. 

\begin{lemma}\label{lem:jose2} 
If $\lambda<\mu$, then 
$C_{s,\mu}\cap S  \subset C_{s,\lambda}\cap S $,
where $C_{s,\mu},C_{s,\lambda}$ are 
minimizers of $\Fcal_{s,\mu}, 
\Fcal_{s,\lambda}$ respectively.
\end{lemma}
\begin{proof}
We know from \cite{AmbFusPal00}, Proposition 3.3.8, that 
\begin{align*}
P(C_{s,\lambda} \cup C_{s,\mu})+P(C_{s,\lambda} \cap C_{s,\mu})\leq 
P(C_{s,\lambda})+P(C_{s,\mu})
\end{align*}
implying
\begin{equation}\label{eq:inequalityAs}
\begin{aligned}
\lambda \lr{P(C_{s,\lambda} \cup C_{s,\mu})-P(C_{s,\lambda})}
&\leq \lambda \lr{P(C_{s,\mu})-P(C_{s,\lambda} \cap C_{s,\mu}}\\
&\leq \mu \lr{P(C_{s,\mu})-P(C_{s,\lambda} \cap C_{s,\mu})}\;.
\end{aligned}
\end{equation}
The last inequality is strict iff $\abs{P(C_{s,\mu})-P(C_{s,\lambda} \cap 
C_{s,\mu})}>0$.

Because of the optimality of $C_{s,\lambda}$ and $C_{s,\mu}$, 
we have 
$\Fcal_{s,\lambda}(C_{s,\lambda}) 
\leq \Fcal_{s,\lambda}(C_{s,\lambda} \cup C_{s,\mu})$, 
$\Fcal_{s,\mu}(C_{s,\mu}) \leq \Fcal_{s,\mu}(C_{s,\lambda} \cap C_{s,\mu})$ 
implying
\begin{align*}
\lambda P\lr{C_{s,\lambda}} 
&+s \abs{C_{s,\lambda} \setminus S}
-(1-s) \abs{C_{s,\lambda} \cap S}\\
&\leq 
\lambda P\lr{C_{s,\lambda} \cup C_{s,\mu}} 
+s \abs{(C_{s,\lambda} \cup C_{s,\mu}) \setminus S}
-(1-s) \abs{(C_{s,\lambda} \cup C_{s,\mu}) \cap S}\\
\mu P\lr{C_{s,\mu}} 
&+s \abs{C_{s,\mu} \setminus S}
-(1-s) \abs{C_{s,\mu} \cap S}\\
&\leq 
\mu P\lr{C_{s,\lambda} \cap C_{s,\mu}} 
+s \abs{(C_{s,\lambda} \cap C_{s,\mu}) \setminus S}
-(1-s) \abs{(C_{s,\lambda} \cap C_{s,\mu}) \cap S}
\end{align*}
such that
\begin{equation}\label{eq:inequalityBs}
\begin{aligned}
\mu &\lr{P(C_{s,\mu})-P(C_{s,\lambda} \cap C_{s,\mu})} 
\\
&\leq \lr{\abs{ C_{s,\mu} \cap S}
-\abs{(C_{s,\lambda} \cap C_{s,\mu}) \cap S}}\\
&+ s \lr{\abs{(C_{s,\lambda} \cap C_{s,\mu}) \setminus S} - \abs{C_{s,\mu} \setminus S}
-\abs{ C_{s,\mu} \cap S}
+\abs{(C_{s,\lambda} \cap C_{s,\mu}) \cap S}
}
\\
&\leq  \lambda \lr{P(C_{s,\lambda} \cup C_{s,\mu}) - P(C_{s,\lambda})}\;.
\end{aligned}
\end{equation}
\eqref{eq:inequalityAs} and \eqref{eq:inequalityBs}
can only hold true if
$\lr{P(C_{s,\mu})-P(C_{s,\lambda} \cap C_{s,\mu})}$
which implies  
\begin{align}
\lr{
 \abs{ C_{s,\mu} \cap S}
-\abs{(C_{s,\lambda} \cap C_{s,\mu}) \cap S}
}
+ 
s 
\lr{
\abs{(C_{s,\lambda} \cap C_{s,\mu}) } - 
\abs{C_{s,\mu} }
}
= 0.
\end{align} 
This in turn implies that

$$\abs{ C_{s,\mu} \cap S}=\abs{(C_{s,\lambda} \cap C_{s,\mu}) \cap S}
\qquad
\abs{(C_{s,\lambda} \cap C_{s,\mu}) } = \abs{C_{s,\mu} }.
$$
With this we conclude the Lemma.
\end{proof}

\begin{lemma}
\label{le:transversalSets}
\begin{enumerate}
\item \label{it:l2} 
For $\lambda, s$ such that $\frac{\lambda}{1-s}\leq r_2$,
there is at most one transversal set.
\item \label{it:l3} 
Assume $s,\mu\leq \lambda$ with $\frac{\lambda}{1-s}\leq
r_2$ are such that there exist $T_{s,\lambda},T_{s,\mu}$ two transversal sets
with $T_{s,\mu} \cap S_2\not = \emptyset$ and $T_{s,\lambda}\cap S_2 \not =
\emptyset$. Then $T_{s,\mu} \subset T_{s,\lambda}$.
\item \label{it:l4} 
If $\lambda < r_2(1-s)$, then $T_{s,\lambda}$ cannot be a minimzer of
$\Fcal_{s,\lambda}$.  
\item \label{it:l1}
The sets $\GammaS$ can be transversal sets of decreasing type, equal to $S_1$ or $\close{\lambda/s}{S}$.
\end{enumerate}
\end{lemma}

\begin{proof}
\begin{enumerate}
\item Assume $s,\lambda$ such that $\rinnen\leq r_2$. 
Set $r_i:=\frac{\lambda}{1-s},r_o:=\frac{\lambda}{s}$.
Assume that $S_1$ and $S_2$ are as in Figure \ref{fig:casesTrans}, that
is, $S_1$ is on the left side of $S_2$ and the centers are located at 
$(-r_1-r_2-d,0)$ and $(0,0)$ respectively.

Moving a circle with radius $r<r_2$ from left to right through $S_1$, 
we observe the following three cases 
\begin{enumerate}[a)]
\item  there is an arc with angular span $\leq \pi$ and the tangents direct towards
$S_1$
\item  there is an arc with angular span $>\pi$ and the tangents direct towards $S_1$
\item  the circle lies inside $S_2$, or the tangents do not point towards
$S_1$.
\end{enumerate} 

\begin{figure}
\begin{center}
\begin{tikzpicture}
 \fill[blue,opacity=0.2,xshift = -3.1 cm] (60:2) arc(60:-60:2);
 \fill[draw=black,fill=blue,opacity=0.2] (0,0) circle (1);
 \def\a{-0.74}
 \fill[white,opacity=0.2] (\a,-0.8) rectangle (-1,0.8);
 \draw[gray!60!] (\a,0) circle (0.7);
 \draw[gray!60!] (\a-0.4,0) circle (0.7);
 \draw[gray!60!] (\a-0.8,0) circle (0.7);
 \draw[blue,xshift = -0.74 cm] (90:0.7) arc(90:-90:0.7);
 \draw[blue!70!,xshift = -1.14 cm] (60:0.7) arc(60:-60:0.7);
 \draw[blue!70!,xshift = -1.54 cm] (30:0.7) arc(30:-30:0.7);
\node at (0,-1.2) {angular span $\leq \pi$};
\end{tikzpicture}
\begin{tikzpicture}
\fill[blue,opacity=0.2,xshift = -3.1 cm] (60:2) arc(60:-60:2);
\fill[draw=black,fill=blue,opacity=0.2] (0,0) circle (1);
\def\a{-0.74}
\fill[white,opacity=0.2] (\a,-0.8) rectangle (-1,0.8);
\draw[gray!60!] (\a,0) circle (0.7);
\draw[gray!60!] (\a+0.4,0) circle (0.7);
\draw[gray!60!] (\a+0.3,0) circle (0.7);
\draw[red,xshift = -0.74 cm] (90:0.7) arc(90:-90:0.7);
\draw[red!70!,xshift = -0.34 cm] (150:0.7) arc(150:-150:0.7);
\draw[red!70!,xshift = -0.44 cm] (120:0.7) arc(120:-120:0.7);
\node at (0,-1.2) {angular span $> \pi$};
\end{tikzpicture}
\begin{tikzpicture}[xscale = -1]
\fill[xscale=-1,blue,opacity=0.2,xshift = -3.1 cm] (60:2) arc(60:-60:2);
\fill[draw=black,fill=blue,opacity=0.2] (0,0) circle (1);
\def\a{-0.74}
\draw[gray!60!] (\a,0) circle (0.7);
\draw[gray!60!] (\a-0.4,0) circle (0.7);
\draw[gray!60!] (\a-0.8,0) circle (0.7);
\draw[violet,xshift = -0.74 cm] (90:0.7) arc(90:-90:0.7);
\draw[violet!70!,xshift = -1.14 cm] (60:0.7) arc(60:-60:0.7);
\draw[violet!70!,xshift = -0.14 cm] (0,0) circle (0.7);
\draw[violet!70!,xshift = -1.54 cm] (30:0.7) arc(30:-30:0.7);
\node at (0,-1.2) {no arc or};
\node at (0,-1.4) {wrong side};
\node[white] at (0,1.4) {wrong side};
\end{tikzpicture}
\end{center}
\caption{Moving a circle with radius $r<r_2$ from left to right through the
circle with radius $r_2$, we observe the following three cases: 
a) there is an arc with angular span $\leq \pi$ and the tangents direct to $S_1$,
b) there is an arc with angular span $>\pi$ and the tangents direct to $S_1$,
c) the circle lies inside $S_2$, or the tangents do not point towards $S_1$. 
}
\label{fig:casesTrans}
\end{figure}
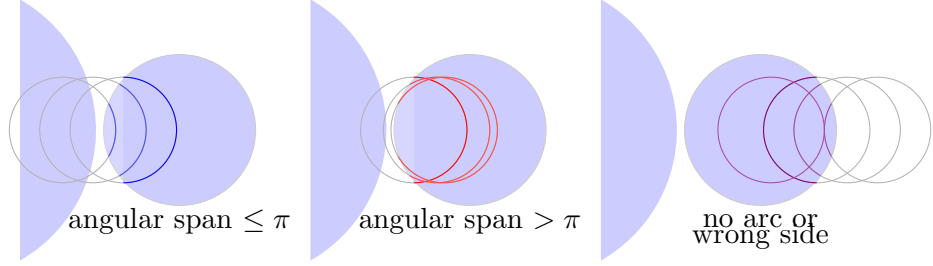

Set $\alpha$ the half angular span  of a circle, centered at 
$\lr{c(\alpha),0}$ intersecting with $\partial S_2$ 
as shown in Figure \ref{fig:TransNotation}. 
Explicitly we have $c(\alpha) = -\lr{r_i\cos\lr{\alpha} + r_2 \cos\lr{\arcsin\lr{\frac{r_i}{r_2}\sin\lr{\alpha}}}}$.
We can restrict our attention to 
circles with centers in $(c(\beta),0)$ for  $\beta \in (0,\pi/2)$ such that 
$c(\beta) \in \lr{-(r_2+r_i),-\sqrt{r_2^2-r_i^2}}$, case a).

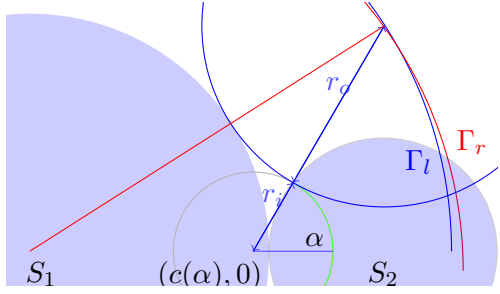
\begin{figure}
\begin{center}
\begin{tikzpicture}[scale = 1.5];
\clip (-3.3,-0.3) rectangle (1,2.2);
\def\r{1.6}
\def\R{2.1}
\def\RR{2.0}
\def\be{34}
\def\al{60}
 
\fill[blue,opacity=0.2,xshift = -3.1 cm] (0,0) circle (\R);

\draw[blue,opacity=1,xshift = -3.1 cm] (\al:\R+\r) arc(\al:-0:\R+\r);
\draw[red,opacity=1,xshift = -3.27 cm] (55:\R+\r) arc(\al:-0:\R+\r); 
 \node[red,xshift = -3.1 *1.5 cm] at (14:\R+\r+0.3) {$\Gamma_r$}
 ;
 \node[blue,xshift = -3.1 *1.5 cm] at (13:\R+\r-0.2) {$\Gamma_l$}
 ;
 
 \fill[draw=black,fill=blue,opacity=0.2] (0,0) circle (1);
 \def\a{-0.74}
 
 \draw[gray!60!] (\a-0.4,0) circle (0.7);
 \draw[green!70!,xshift = -1.14 cm] (\al:0.7) arc(\al:-\al:0.7);

 \draw [blue!70!,xshift = -1.14 cm] (0,0) -- (\al:0.7+\r) node (a) {};
 \draw [blue!70!,xshift = -1.14 cm,<->] (0,0) -- (\al:0.7) node[midway,above] (a2) {$r_i$};
 \draw [blue!70!,xshift = -1.14 cm,<->] (\al:0.7) -- (\al:0.7+\r) node[midway,above] (a3) {$r_o$};
 
 \draw [blue!70!,xshift = -1.14 cm] (0,0) -- (0:0.7);
 \node at (0,-0.2) {$S_2$};
 \node at (-3,-0.2) {$S_1$};
 \node [xshift = -2.28 cm,below] at (0,0) {$(c(\alpha),0)$};
 \node [] at (-0.6,0.1) {$\alpha$};
  \draw [blue,xshift = -1.14  cm] (\al:0.7+\r) circle (\r);
  \draw [blue,xshift = -1.14  cm] (0,0) -- (\al:0.7+\r);
  \draw[red, xshift = -1.14  cm] (\al:0.7+\r) -- (-3.1+1.14,0);
\end{tikzpicture}

\end{center}
\caption{Left: if $\Gamma_r,\Gamma_l$ intersect once in the first quadrant, there is one unique transversal set.
Right: Case of two transversal sets (inner and outer).}
\label{fig:TransNotation}
\end{figure}
To show that there exists maximal one transversal set in the case where 
$r_i\leq r_2$ we use the following construction:
Set 
\begin{align*}
\gamma_l(\beta)&:=
(r_1+r_o )
\begin{pmatrix}
\cos(\beta)\\ \sin(\beta)
\end{pmatrix}
\\
\gamma_r(\beta)&:=
\begin{pmatrix}
c(\alpha)+(r_i+r_o)\cos(\alpha)\\
(r_i+r_o)\sin(\alpha))
\end{pmatrix}
\end{align*}
and $\Gamma_l:=\set{\gamma_l(\beta),\beta \in (0,\pi/2)}, 
\Gamma_r:=\set{\gamma_r(\alpha),\alpha \in (0,\pi/2)}$.

$\Gamma_l$ contain the centers of the circles with radius $\frac{\lambda}{s}$ that are tangential to 
$\partial S_1$. 
$\Gamma_r$ contains the centers of circles that area tangential to the arcs 
inside $S_2$ at the intersection point with $\partial S_2$.

For a transversal set $\Gamma_{s,\lambda}$, the center of the arc 
$\partial \Gamma_{s,\lambda}\setminus S$ in the positive $y$-plane  must be an element of 
$\Gamma_s \cap \Gamma_l$ (condition to be smooth).

$\Gamma_l,\Gamma_r$ can be written as to striclty increasing, concave functions, hence they 
intersect at most once, i.e. 
there is only one set $(\beta_0,\alpha_0)$ with 
$\gamma_l(\beta_0) = \gamma_r(\alpha_0)$. This implies that there is at most one transversal set.

\item Denote by $\alpha_\mu,\beta_\mu$ the angular span the connected component
of $\partial T_{s,\mu} \cap S_2, \partial T_{s,\mu}\setminus S$, respectively.
Moreover denote by $p_1,p_2$ the intersection points of $\partial T_{s,\mu}$ and
$\partial S_2$. 
An arc inside $S_2$ with radius $\frac{\lambda}{1-s}$ intersecting
with $\partial S_2$ at $p_1,p_2$ has an angular span denoted by $\alpha_\lambda$
that is smaller than $\alpha_\mu$. 
Now if we continue smoothly with an arc with radius $\frac{\lambda}{s}$ at
$p_1,p_2$, then this arc will not intersect with $\partial S_1$.
In order to increase the angular span $\alpha_\lambda$, we have to move the arc
inside $S_2$ outside of $T_{\mu,s}$ until the arc outside of $S_2$ intersects
with $\partial S_1$. Hence at the end we see that $T_{\mu,s}\cap S \subset
T_{\lambda,s}\cap S$.
\begin{figure}
\begin{center}
\begin{tikzpicture}[scale = 3]

\clip (0.0,-0.4) rectangle (2.1,1.35);
\draw[fill=blue,fill opacity=0.2] (0,0) circle (1);
\draw[fill=blue,fill opacity=0.2] (1.6,0) circle (0.5);

\def\a{0.345}
\draw[red,thick] (40:1) arc(180+40:255:0.95) arc (75:-75:\a)
arc(90+15:180-40:0.95);

\draw[thick,blue] (43:1) 
arc(180+43:257:1.1) node[below] (a) {$p_1$}
arc (77:-77:0.37) node [above] (b) {$p_2$}
arc(90+15:180-40:1.19);
\node[right] at (1.17,0.1) {$\frac{\alpha_{\mu}}{2}$};
\node[right] at (0.89,0.6) {$\beta_{\mu}$};
\draw[] (40:1) -- (40:1.95) -- ++(255:0.95+\a)-- ++(0:\a);

\draw[xshift = 1.04 cm,violet,very thick,dotted] (-60:0.4)  arc(-60:60:0.4)
arc(180+60:90:1.1);

\fill (1.24,+0.34) circle (0.02);
\fill (1.24,-0.34) circle (0.02);
\end{tikzpicture}
\caption{Increasing the radius of the ball touching the points $p_1,p_2$, the smooth continuation with a circle, does not toucht $S_1$ anymore. Moving the circle with larger radius to the right, the outer circle touches again $S_1$.
}
\end{center}
\end{figure}
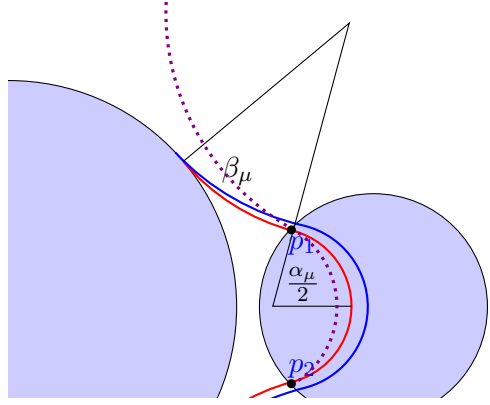

\item From the previous two items we know that if $\lambda \leq r_2(1-s)$, a transversal set is of increasing type. 
Lemma \ref{lem:jose2} states that transversal sets of increasing type cannot be a minimizers.
\item \changed{Follows from the previous 3 items, together with Lemma \ref{lem:jose2}, that states that transversal sets of increasing type cannot be minimizers of $\Fsl$.}
\end{enumerate}
\end{proof}

\section{The dual norm of $\chi_S$}\label{sect:dualnorm}

In this section we compute the dual norm $\|\chi_S\|_*$.
If $g=\chi_E$, then
$\norm{\chi_E}_\ast \leq \mu$
if and only if
\begin{equation*}
 0\leq \min_{F\subset \R^2} \set{P(F) - \frac{1}{\mu}\abs{F\cap E}} 
   = \min_{F\subset \R^2} \Fcal_{\mu}\lr{F}\;.
\end{equation*}
Let $C_{0,\lambda}$ minimize $\Fcal_{0,\lambda}$, then 
if $\lambda = \frac{\abs{C_{0,\lambda}\cap S}}{P(C_{0,\lambda})}$, 
$\lambda$ is the minimal parameter such that $u= 0$ minimizes \eqref{vpI}
and $\norm{\chi_{S}}_* = 
\frac{\abs{C_{0,\lambda}\cap S}}{P(C_{0,\lambda})}$. 
Hence, letting $\rho(X):=\frac{\abs{X \cap S}}{P(X)}$ for $X\subset\R^2$, we have
\begin{equation}\label{eqmax} 
\norm{\chi_{S}}_* = \max_{X\subset\R^2} \rho(X)\,.
\end{equation}

\begin{proposition}\label{pro:gnorm}
We have
\begin{equation*}
 \norm{\chi_S}_{*} = \max \set{
\frac{\abs{S_1}}{P(S_1)},\frac{\abs{S}}{P(co(S))}}\;.
\end{equation*}
\end{proposition}

For $r\in (0,r_1)$ we let $S(r):=S_1 \cup S_2(r)$,
where $S_2(r)$ is a ball of radius $r$ centered at the center of $S_2$.
Before proving the proposition, we need the following Lemma:

\begin{lemma}\label{le:smallerr_2}
If $S_1$ maximizes 
$X\to \frac{\abs{X\cap S(R)}}{P(X)}$
over all $X\subset \R^2$, for some $R>0$, then it also
maximizes $X\to \frac{\abs{X\cap S(r)}}{P(X)}$, for every $0<r < R$. 
\end{lemma}
\begin{proof}
Note that if $X$ maximizes 
\begin{equation}\label{eq:setmax} 
X\to \frac{\abs{X\cap S(r)}}{P(X)}
\end{equation}
then
$X$ minimizes $\Fcal_{0,\lambda}$ for $\lambda = 
\frac{\abs{X \cap S(r)}}{P(S(r))}$. 
Hence by Proposition \ref{lemma:gg} we know that any possible maximizer
of \eqref{eq:setmax} must be in 
\begin{equation*}
 \set{S_1, S(r), co(S(r)), T_{\lambda}^+(S(r))}\;.
\end{equation*}
Since by assumption $r \leq r_1$, we can exclude $S(r)$ since 
\begin{align*}
\frac{\abs{S_1}}{P(S_1)}&=\frac{r_1}{2}\geq \frac{r^2+r_1^2}{2(r+r_1)} =
\frac{\abs{S(r)}}{P(S(r))}\,
\;.
\end{align*}
It remains to exclude transversal sets. Assume by contradiction
that we can find $X_2 \subset \R^2$ with 
$X_2 \cap S_2(r) \not \in \set{\emptyset, S_2(r)}$, 
that  maximizes
\eqref{eq:setmax}. Then 
 $\frac{ \abs{S_1}} {P(S_1)} \leq \frac{\abs{X_2 \cap S(r)}}{P(X_2)}$.

Since by assumption $S_1$ 
maximizes $X\to \frac{\abs{X\cap S(R)}}{P(X)}$, we get 
$\frac{ \abs{X_2 \cap S(R)}}{P(X_2)} \leq \frac{\abs{S_1}}{P(S_1)}$.
Both conditions together yield
\begin{equation*}
\frac{ \abs{X_2 \cap S(R)} }{P(X_2)} 
\leq  \frac{\abs{S_1}}{P(S_1)}
\leq  \frac{\abs{X_2 \cap S(r)}}{P(X_2)}
\end{equation*}
such that $\abs{X_2 \cap S(R_2)} \leq \abs{X_2 \cap S(r_2)}$ contradicting 
the assumption $r_2 < R_2$ ($S(r_2)\subset S(R_2)$), hence we conclude  the
statement.
\end{proof}

Now we are ready to prove Proposition \ref{pro:gnorm}:

\begin{proof}
Set $\hat r_2$ as the radius larger than zero, such
that $\frac{\abs{S_1 \cup S_2(\hat r_2)}}{P(co(S(\hat r_2))} =
\frac{\abs{S_1}}{P(S_1)}$. 
It is basic calculus to proof that $\hat r_2$ exists and for 
$r_2<\hat r_2$ $\frac{\abs{S_1 \cup S_2(r_2)}}{P(co(S(r_2))} >
\frac{\abs{S_1}}{P(S_1)}$  and 
$\frac{\abs{S_1 \cup S_2(r_2)}}{P(co(S(r_2))} <
\frac{\abs{S_1}}{P(S_1)}$ for $r_2 > \hat r_2$ (see Figure \ref{fig:diagram}). 

Consider the following cases:
\begin{itemize}
\item $r_2>\hat r_2$.  
Set $\lambda:=\frac{\abs{S_1 \cup S(r_2)}}{P(co(S(r_2))}$. 
Then $\lambda<r_2$ 
(see Figure \ref{fig:diagram}). 
This implies that there is no 
transversal set minimizing $\Fcal_{0,\lambda}$. 
The only choices for minimizers are $S_1$ and $co(S)$, since we assume
$r_2>\hat r_2$, we have
$\Fcal_{0,\lambda}(S_1)>\Fcal_{0,\lambda}(co(S)) = 
\Fcal_{0,\lambda}(\emptyset)=0$, hence $co(S)$ is the optimal set and 
$co(S)$ maximizes \eqref{eq:setmax}.

\item Case $r_2=\hat r_2$ analog to the previous case, but 
$\Fcal_{0,\lambda}(S_1)=\Fcal_{0,\lambda}(co(S)) = 
\Fcal_{0,\lambda}(\emptyset)=0$, hence $co(S), S_1$ are the optimal sets and 
$S_1, co(S)$ maximize \eqref{eq:setmax}.

\item Case $r_2<\hat r_2$. 
We can find $\epsilon>0$ such that for 
For $r_2= \hat r_2 -\epsilon$,  
$\frac{ \abs{S_1 \cup S_2(r_2)} }{ P(co(S(r_2)) }<r_2$. 
Hence also in this case there would be no optimal transversal set. 
Moreover we have $\frac{ \abs{S_1 \cup S_2(r_2)} }{ P(co(S(r_2)) }
<\frac{\abs{S_1}}{P(S_1)}$, hence by Lemma \ref{le:smallerr_2}, we can conclude
that for all $r_2<\hat r_2 - \epsilon$, $S_1$ maximizes $X\rightarrow
\frac{\abs{X\cap S(r_2)}}{P(X)}$.
\end{itemize}
Hence we conclude that the only possible choices for an optimal set of
\eqref{eq:setmax} are $S_1$ and $co(S(r_2))$ and conclude the Lemma.
\end{proof}

\begin{figure}
\label{fig:diagram}
\begin{center}
\includegraphics[width=15cm]{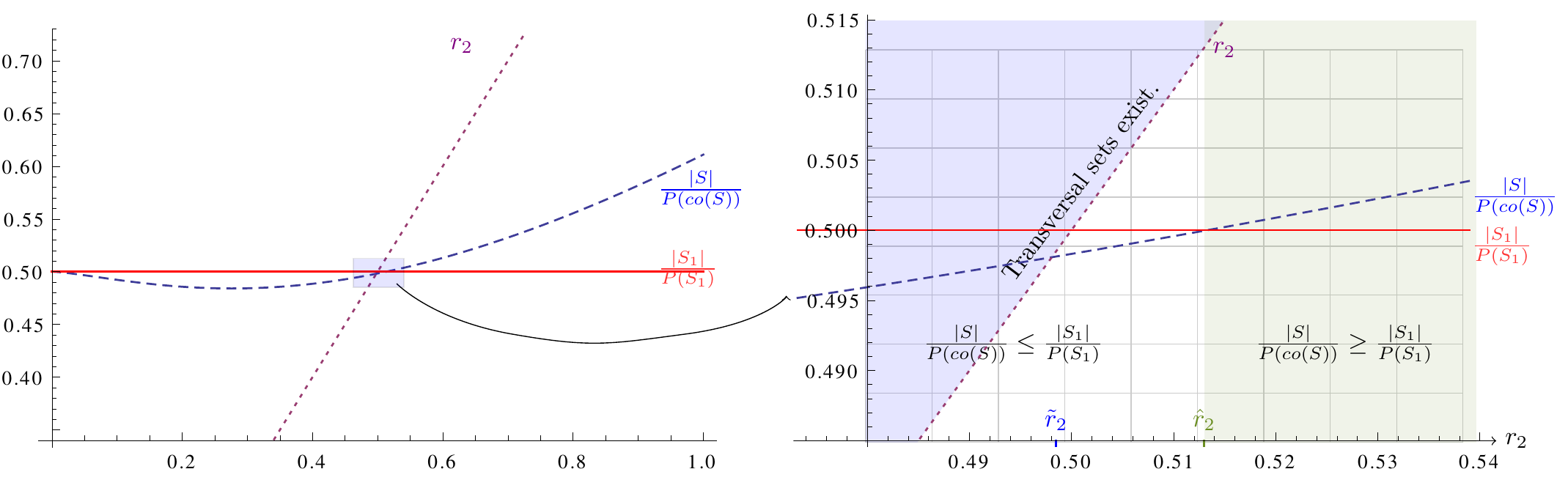}
\end{center}
\caption{ We fix $r_1=1$, vary $r_2$ and look at the ratios of
$\frac{S_1}{P(S_1)}$ and $\frac{S}{P(co(S))}$. 
Solid line $\frac{\abs{S_1}}{P(S_1)}$, dashed line
$\frac{\abs{S(r_2)}}{P(co(S(r_2)))}$ for different $r_2$, dotted line $r_2$.
Left: general situation, right: zoom around $r_2=0.5$.}
\label{fig:diagram}
\end{figure}


\section{Construction of the minimizers $C_{s,\lambda}$}\label{sect:construction}
\begin{construction} \label{co:GammaStar}
Let $\lambda>0$ and $s\in [0,1]$.
Understand $\close{\frac{\lambda}{0}}{S}$ as $co(S)$ (convex hull
of $S$).
\begin{enumerate}
\item Set $k:=0$,  $\Xsl^0:=S$, $\Ysl^0:=\close{\raussen}{S}$.
\item 
For $k=k+1$  set 
\begin{equation*}
\Xsl^k:=\open{\rinnen}{\Ysl^{k-1}}, 
\end{equation*}
and
\begin{equation*}
\Ysl^k:=\close{\raussen}{\Xsl^k \cap S}\;.
\end{equation*}

\item Finally define $\GammaS:= \bigcap_{k\in \N} \Ysl^k$
$(=\lim_{k\rightarrow \infty}\Ysl^k)$.
\end{enumerate}
\end{construction}

\begin{remark}\rm 
The sets $\Xsl^k$ and $\Ysl^k$ have the following properties:

\begin{enumerate}[i)]
\item   $\Xsl^0\subset \Ysl^{0}$ 
          and $\Ysl^0\cap S \subset \Xsl^0$ such that
         $\col{\Xsl^1} = \open{\rinnen}{\Ysl^0 \cap S} \subset
\open{\rinnen}{\Xsl^0} \subset 
         \col{\Xsl^0}$.
         Consequently $\Ysl^0 \subset \Ysl^1$ and so on. In general we have
	\begin{align*}
	 \Ysl^{k+1} \subset \Ysl^{k} \quad \text{ and } \quad
        \Xsl^{k+1} \subset \Xsl^{k}\;.
	\end{align*}
\item Due to the properties of the opening and the closing operators (see Lemma
\ref{le:propertiesOpeningClosing}) we have for the curvature $\kappa$: 
	 \begin{align*}
	  \kappa(\partial \Xsl ^k) &\leq \curvin 
&\text{and} &
	  &\kappa( \partial \Ysl^k) &\geq \curvout \;.
	  \end{align*}
\item If $\Ysl^1 = \Xsl^1$, then $\Ysl^k=\Ysl^1$ for all $k>1$ such that
        $\GammaS=\Ysl^1$. 
\item  In the case where $\Xsl^k \not = \Xsl^{k+1} =\open{\rinnen}{
\close{\raussen}{\Xsl^k \cap S}}
          = \open{\rinnen}{\Ysl^k}$, 
    there exists a part in $\partial \Ysl^k $ with curvature  larger than 
      $\curvin$. Applying another opening  to $\Ysl^{k}$ we replace this part,
but then
      $\partial \open{\rinnen}{ \Ysl^{k}}\setminus S$ might have parts with
 curvature different from 
       $-\frac{s}{\lambda}$. 
\item For every $k$ we have
       $
       \open{\rinnen}{S} \subset \Xsl^k, \Ysl^k \subset \close{\rinnen}{S}\;.
       $
\end{enumerate}
\end{remark}

\begin{remark}\rm
The sets $\GammaS$ have the following properties:
\begin{enumerate}[a)]
 \item  
 $\open{\rinnen}{\GammaS} \cap S = \GammaS \cap S$. 
\item $\close{\raussen}{\GammaS}=\GammaS$.
 \item $\partial \GammaS
\cap S$ has curvature $\frac{1-s}{\lambda}$, and $\partial \GammaS
\setminus S$ has curvature $-\frac{s}{\lambda}$. 

\item \label{item:boundaries} 
$\partial \GammaS$ is of class $\mathcal{C}^{1,1}$ and 
$-\frac{s}{\lambda} \leq \kappa\lr{\partial
\GammaS} \leq \frac{1-s}{\lambda}$.
\item If $\rinnen\leq r_2\leq r_1$ then $\GammaS=\close{\raussen}{S}$.
\item For $r_2<\rinnen \leq r_1$ then 
$S_1\subset\GammaS\subsetneq\close{\raussen}{S}$. 
\item For $\rinnen > r_1$,  $\GammaS = \emptyset$. 
\end{enumerate}
\end{remark}

We now give an explicit characterization of the solutions
 of \eqref{vpI}.
\begin{remark}\rm
By Lemma \ref{destruccio},
we may assume that
\begin{equation}\label{contextenodestruccio}
P(S) > P(\mathrm{co}(S)),
\end{equation}
otherwise the solution corresponding to $S$ is described by the
sum of the solutions corresponding to $S_1$ and $S_2$, that is, both sets do not
interact.
Moreover, by Lemma \ref{pro:gnorm} and Proposition \ref{p2} it is enough to consider
\begin{equation*}\label{eq:aquesta12}
\lambda \leq \max\left
\{\frac{|S_1|}{P(S_1)},
\frac{|S|}{P(\mathrm{co}(S))}\right
\},
\end{equation*}
otherwise  the solution of \eqref{vpI} is equal to zero.
\end{remark}
We start by comparing the energies of $S_1,S_2,S$ and
$\emptyset$. 

\begin{lemma}\label{le:comparScircle}
Let $\lambda  > 0$. We have
\begin{align*}
\min \set{0 , \Fsl\lr{S},\Fsl\lr{ S_1}, \Fsl\lr{S_2}}=
\begin{cases}
\Fsl\lr{ S}
&\frac{P(S_2)}{\abs{S_2}}\leq \frac{1-s}{\lambda},\\
\Fsl\lr{ {S_1}}
& \frac{P(S_1)}{\abs{S_1} }\leq \frac{1-s}{\lambda} \leq
\frac{P(S_2)}{\abs{S_2}},\\
0
& \frac{1-s}{\lambda} \leq \frac{P(S_1)}{\abs{S_1}}.
\end{cases}
\end{align*}
\end{lemma}

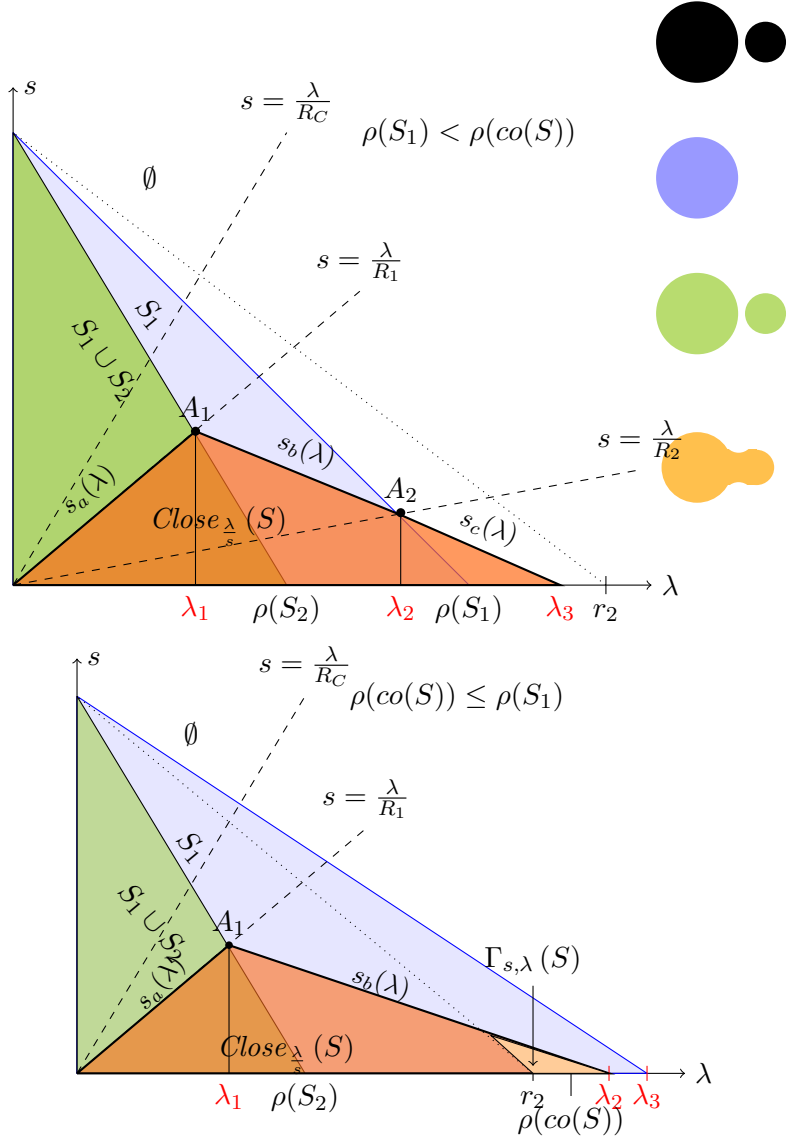
\begin{figure}
\begin{center}
\begin{tikzpicture}[scale=6]
\node at (1,1) {$\rho(S_1) <\rho(co(S))$};
\draw[->] (0,0) -- (1.4,0) node[right] {$\lambda$};
\draw[->] (0,0) -- (0,1.1) node[right] {$s$};
\filldraw[blue,opacity=0.1, blue,draw opacity=1] (0,1) --(0,0) -- (1,0)
node[below,black,opacity=1] {$\rho(S_1) $} -- (0,1);
\filldraw[YellowGreen,opacity=0.7,draw opacity=1,draw=black] (0,1)  --
(0,0)--(0.6,0) node [below,black,opacity=1] {$\rho(S_2)$} -- (0,1);
\filldraw[Red!70!Yellow!80!Orange!90!,opacity=0.7, draw
opacity=1,draw=black,thick] 
(0,0)  -- (0.4,0.34) 
node[midway,black, opacity=1,rotate=40,above] {\small{$s_a(\lambda)$}}
--(0.85,0.155) 
node[midway,black, opacity=1,rotate=-20,above] {\small{$s_b(\lambda)$}}
 -- (1.2,0)  node[midway,black, opacity=1,rotate=-20,above]
{\small{$s_c(\lambda)$}}
-- (1.2,0) node (a3) [below,red,opacity=1,] {$\lambda_3$}
 -- (0.85,0)  node (a2) [below,red,opacity=1] {$\lambda_2$} 
-- (0.4,0) node[below,red,opacity = 1] (a1)   {$\lambda_1$} --(0,0);
\draw (a1) -- (0.4,0.34);
\draw (a2) -- (0.85,0.155);

\node[rotate=-50] at (0.3,0.6) {$S_1$};
\node[rotate=-55] at (0.2,0.5) {$S_1 \cup S_2 $};

\node at (0.3,0.9) {$\emptyset$};
\node at (0.45,0.13)
{$\close{\raussen}{S}$};

 \begin{scope}[scale=1.5]
 \fill[black] (1.0,0.8) circle (0.06);
 \fill[black] (1.1,0.8) circle (0.03);
 \end{scope}
 
 \begin{scope}[scale=1.5]
 \fill[blue!40!] (1.0,0.6) circle (0.06);
 \end{scope}

  \begin{scope}[scale=1.5]
 \fill[YellowGreen!70!] (1.0,0.4) circle (0.06);
 \fill[YellowGreen!70!] (1.1,0.4) circle (0.03);
 \end{scope}
  \begin{scope}[xshift=0.2cm, scale=1.3]
 \fill[Orange!70!] (1.0,0.2) circle (0.06);
 \fill[Orange!70!] (1.1,0.2) circle (0.03);
  \fill[Orange!70!] (1.0,0.23) rectangle (1.1,0.17);
  \fill[white] (1.07,0.25) circle (0.025);
    \fill[white] (1.07,0.15) circle (0.025);
 \end{scope}

\fill (0.4,0.34) node[above] {$A_1$} circle (0.01);
\fill (0.85,0.16) node[above] {$A_2$} circle (0.01);

\draw[dotted] (1.3,0) -- (0,1);

\draw[dashed](0,0) -- (0.6,1) node[above] {$s=\frac{\lambda}{R_C}$};

\draw[dashed](0,0) -- (10.5:1.4) node[above] {$s=\frac{\lambda}{R_2}$};
\draw[dashed](0,0) -- (40.5:1) node[above] {$s=\frac{\lambda}{R_1}$};

\draw (1.3,0.02) -- (1.3,-0.02) node[below] {$r_2$};
\end{tikzpicture}

\begin{tikzpicture}[scale=5]
\node at (1,1) {$\rho(co(S)) \leq \rho(S_1) $};
 \draw[->] (0,0) -- (1.6,0) node[right] {$\lambda$};
 \draw[->] (0,0) -- (0,1.1) node[right] {$s$};

  \filldraw[blue,opacity=0.1, blue,draw opacity=1] (0,1) --(0,0) -- (1.5,0)
node[below, red,opacity=1] (l3) {$\lambda_3 $} -- (0,1);

\filldraw[YellowGreen,opacity=0.5,draw opacity=1,draw=black] (0,1)  --
(0,0)--(0.6,0) node [below,black,opacity=1] {$\rho(S_2)$} -- (0,1);

\filldraw[Red!70!Yellow!80!Orange!90!,opacity=0.6, draw
opacity=1,draw=black,thick] (0,0)  -- (0.4,0.34) 
--(1.4,0.0)  
node (a2) [below,red,opacity=1] {$\lambda_2$}  
 -- (0.4,0) node[below,red,opacity = 1] (a1)   {$\lambda_1$} --(0,0);
\draw (a1) -- (0.4,0.34);
\draw[red] (1.4,0.02) -- (1.4,-0.02); 
\draw[red] (1.5,0.02) -- (1.5,-.02); 

\draw (1.3,0.01) -- (1.3,-0.06) node[below] {$\rho(co(S))$};
\draw (1.2,0.01) -- (1.2,-0.02) node[below] {$r_2$};

\draw[dotted] (1.2,0) -- (0,1); 

\fill[orange!40!,draw = black] (1.2,0) -- (1.09,0.1) -- (1.4,0) -- (1.2,0);

\node[rotate = 50] at (0.22,0.24)  {\small{$s_a(\lambda)$}};
\node[rotate = -17] at (0.8,0.24) {\small{$s_b(\lambda)$}};

\node[rotate=-50] at (0.3,0.6) {$S_1$};
\node[rotate=-50] at (0.2,0.4) {$S_1 \cup S_2 $};
\node at (0.3,0.9) {$\emptyset$};
\node (b) at (1.2,0.3) {$\Gamma_{s,\lambda}\lr{S}$};
\draw[->] (b) -- (1.2,0.02);
\node at (0.55,0.05) {$\close{\raussen}{S}$};
\fill (0.4,0.34) node[above] {$A_1$} circle (0.01);
\draw[dashed](0,0) -- (40.5:1) node[above] {$s=\frac{\lambda}{R_1}$};
\draw[dashed](0,0) -- (0.6,1) node[above] {$s=\frac{\lambda}{R_C}$};
\end{tikzpicture} 

\caption{
Illustration of Theorem \ref{pr:explicitCircles}.
For any point $(\lambda,s)$ this diagram shows which one of the four sets
$S_1,S,\GammaS,\emptyset$ has the minimal $\Fsl$ value.
For $(\lambda,s) =(\lambda,1-\lambda \frac{P(S_1)}{\abs{S_1}}),
(\lambda,1-\lambda\frac{P(S_2)}{\abs{S_2}})$, the minima are not unique.
The dotted line indicates the values of $(s,\lambda)$ such that
$\rinnen = r_2$, hence there are no transversal minimizers on the left
of the dotted line.
}
\end{center}
\end{figure}

\begin{proof}
Observe that, for $i=1,2$,
 \begin{align*}
 \Fsl\lr{S_i}&=P(S_i)-\frac{1-s}{ \lambda} \abs{S_i} \leq 0 \quad \hbox{\rm if
and only if}
 \quad \frac{1-s}{\lambda}\geq  \frac{P(S_i)}{\abs{S_i}}.
 \end{align*}
 If $\frac{1-s}{\lambda} \geq \frac{P(S_2)}{\abs{S_2}}$, 
 then $\Fsl\lr{S}=\Fsl\lr{S_1}+\Fsl\lr{S_2}\leq \min \set{ 0,\Fsl\lr{ S_1},
\Fsl\lr{S_2} }$.
If $\frac{P(S_1)}{\abs{S_1} }\leq \frac{1-s}{\lambda} <
\frac{P(S_2)}{\abs{S_2}}$, then
$\Fsl\lr{S_2}> 0$ and $\Fsl\lr{S}=\Fsl\lr{S_1}+\Fsl\lr{S_2} > \Fsl\lr{S_1}$.
In case $\frac{1-s}{\lambda} < \frac{P(S_1)}{\abs{S_1}}$, $\Fsl\lr{ S_1}>0$, $
\Fsl\lr{S_2} > 0$
and $\min \set{0 , \Fsl\lr{S},\Fsl\lr{ S_1}, \Fsl\lr{S_2}}=0$.
\end{proof}

We now define special values of $\lambda$
which will be useful  in order to classify the minimizers of
$\Fsl$.

\begin{proposition}\label{de:lambda}
Assume that \eqref{contextenodestruccio} holds.
Let $R_c(S)$  be the minimal radius $r$ such that $\close{r}{S}$ is connected.
\begin{enumerate}[(a)]
\item \label{it:lambdaa} There is a unique value $R_1\in [R_c(S),\infty)$ such that
\begin{align}\label{eq:R1}
 P\lr{\close{R_1}{S}}+\frac{1}{R_1}\abs{ \close{R_1}{S} \setminus S} &=P(S).
 \end{align}
Let $s_2(\lambda) = 1 - \lambda \frac{P(S_2)}{|S_2|}$ and $\lambda_1$ be given
by
$\frac{\lambda_1}{s_2(\lambda_1)}  = R_1$. Then $\lambda_1 :=
\frac{R_1\abs{S_2}}{R_1 P(S_2)+\abs{S_2}} \in \left[0,\fraczwei \right]$ and
\begin{equation}\label{eq:lambda1Circles}
\Fcal_{s_2(\lambda),\lambda}\lr{\close{\frac{\lambda}{s_2(\lambda)}}{S}} >
\Fcal_{s_2(\lambda),\lambda}\lr{S}
\qquad \hbox{\rm (resp. $=$, $<$)}
\end{equation}
for any $\lambda < \lambda_1$ (resp $=$, $>$).
The value $\lambda_1= 0$ if and only if $R_1=0$, and this happens if and only if
$R_c(S)=0$, in other words if the two balls touch eachother. 
If $R_c(S) > 0$, then $R_1 > R_c(S)$.

\item\label{it:lambdab}
\begin{enumerate}[i)]
	\item
		If $\fraceins < \fracco$, there is a unique value $R_2\in 
[R_c(S),\infty)$ such that
		\begin{align}\label{eq:R2}
		P\lr{\close{R_2}{S}}+\frac{1}{R_2}\abs{ \close{R_2}{S}\setminus
S} &=P(S_1)\frac{\abs{S} }{\abs{S_1} }.
		\end{align}
Let $s_1(\lambda) := 1 - \lambda \frac{P(S_1)}{|S_1|}$ and 
$\lambda_2  :=
\frac{R_2\abs{S_1}}{R_2 P(S_1)+\abs{S_1}} \in \left[\lambda_1,\fraceins\right]$.
Then $\frac{\lambda_2}{s_1(\lambda_2)}  = R_2$
and
\begin{equation}\label{eq:lambda2Circles}
\Fcal_{s_1(\lambda),\lambda}\lr{\close{\frac{\lambda}{s_1(\lambda)}}{S}} >
\Fcal_{s_1(\lambda),\lambda}\lr{S_1}
\qquad \hbox{\rm (resp. $=$, $<$)}
\end{equation}
for any $\lambda < \lambda_2$ (resp $=$, $>$).
We have that $R_1=R_2$ if and only if $\frac{P(S_1)}{|S_1|} =
\frac{P(S_2)}{|S_2|}$ if and only if
$\lambda_1 = \lambda_2$. And $R_2 = 0$ (in that case also $R_1=0$ and $\lambda_1
= \lambda_2=0$) if and only if $R_c(S) = 0$.

\item If $\fracco \leq \fraceins$ set $\lambda_2$ as the solution of 
$\lambda_2=\frac{\abs{\Gamma_{0,\lambda_2}\lr{S}\cap
S_2}}{P(\Gamma_{0,\lambda_2 } \lr {S})-P(S_1)}$.

If $\frac{\abs{S_2}}{P(co(S))-P(S_1)}\leq r_2$, then
$\lambda_2=\frac{\abs{S_2}}{P(co(S))-P(S_1)}$.
\end{enumerate}

\item \label{it:lambdac} Set $\lambda_3:=\max\set{\fraceins,\fracco}$.
\end{enumerate}
\end{proposition}

\begin{remark}{\rm
Observe that when $\frac{|S_{1}|}{P(S_{1})} =  \frac{|S|}{P(\mathrm{co}(S))}$,
$\lambda_2 = \frac{|S_{1}|}{P(S_{1})} = \lambda_3$.
}\end{remark}

Now we are ready to describe the minimizers of $\Fsl$. For simplicity, from now on
we denote
by $C_{s,\lambda}$ the largest minimizer of $\Fsl$ (see Proposition \ref{cc}).

\begin{theorem}\label{pr:explicitCircles}
Assume that \eqref{contextenodestruccio} holds.
Let $\lambda_1,\lambda_2,\lambda_3$ be as in Proposition \ref{de:lambda}.
Then the sets $\Cs$ are given by:
\begin{enumerate}[(a)]
\item Let $\lambda \in [0,\lambda_1]$. There is a value $0 < s_a(\lambda)\leq
1-\lambda\fraczweii$ such that

\begin{equation*}
\Cs=
\begin{cases}
\close{\raussen}{S} &0\leq s \leq s_a(\lambda)\\
S &s_a(\lambda)< s \leq 1-\lambda\fraczweii\\
S_1 & 1-\lambda\fraczweii< s \leq 1-\lambda\fraceinsi\\
\emptyset &1-\lambda\fraceinsi <s.
\end{cases}
\end{equation*}
The third interval is empty in the case $\fraceins=\fraczwei$.

\item Let $\lambda \in (\lambda_1,\lambda_2]$.
There is a value 
$1-\lambda\fraczwei < s_b(\lambda) \leq  1-\lambda\fraceinsi $
such that
\begin{equation*}
\Cs=
\begin{cases}
\GammaS &0\leq s \leq s_b(\lambda)\\
S_1 &s_b(\lambda)< s \leq 1-\lambda\fraceinsi\\
\emptyset &1-\lambda\fraceinsi <s \, , 
\end{cases}
\end{equation*}
and $\GammaS = \close{\raussen}{S}$ as long as $\frac{\lambda}{1-s}\leq r_2$. 

\item  Let $\lambda \in (\lambda_2,\lambda_3]$.

\begin{enumerate}[]
\item[(c1)] If $\fraceins<\fracco$, then there is a value $ s_c(\lambda) > 1-\lambda
\frac{P(S_1)}{|S_1|}$
such that

\begin{equation*}
\Cs=
\begin{cases}
\close{\raussen}{S} &0\leq s \leq  s_c(\lambda) \\ 
\emptyset &\text{else}.
\end{cases}
\end{equation*}

\item[(c2)] If $\fracco \leq \fraceins$, then
\begin{equation*}
\Cs=
\begin{cases}
S_1 &0\leq s \leq 1-\lambda \fraceinsi \\
\emptyset &\text{else}.
\end{cases}
\end{equation*}
\end{enumerate}

\item[(d)] For $\lambda>\lambda_3$ $\Cs = \emptyset$. 
\end{enumerate}

\end{theorem}

Figure \ref{fi:solution} shows solutions of \eqref{vpI} for different $\lambda$,
when $S$ is the union of two balls with
radii $r_1=1.2, r_2=1$ and distance $d=0.05$ (case
$\frac{\abs{S_1}}{P(S_1)}<\frac{\abs{S}}{P(co(S))}$). 

In order to prove Proposition \ref{de:lambda} and Theorem \ref{pr:explicitCircles} we need
the following Lemmas:

\begin{lemma} \label{le:arcisbest}
Let $0<R<r$ and define

\begin{equation}\label{eq:Gr}
 G_{r}(h):=\int_{-R}^{R} \lr{\sqrt{1+h'(x)^2}- \frac{1}{r} h(x)}\;dx.
\end{equation}
The function that represents an arc of a circle  with radius $r$ (angular span
smaller than $\pi$) from $-R$ to $R$,
minimizes $G_r(h)$ under all functions $h$ with $h(-R)=0, h(R)=0$.
\end{lemma}

\begin{proof}
See \cite[Lemma 4.29]{ACMBook}.
\end{proof}

\begin{lemma}
\label{le:closingisbest}
Let $R_c(S)$  be the minimal radius $r$ such that $\close{r}{S}$ is connected.
Then for $s\in [0,1], 0<\mu< \lambda$,
such that $R_c(S) \leq \frac{\mu}{s}$,
\begin{equation*}
\Fsl\lr{\close{\frac{\lambda}{s}}{S}}
\leq\Fsl\lr{\close{\frac{\mu}{s} }{S}}\,.
\end{equation*}
\end{lemma}

\begin{proof}
Let us take the $x$-axis as the axis joining the centers of the two circles.
Then $S$ is symmetric with respect to the $x$-axis.
Then the upper parts of
$\partial \lr{\close{\frac{\lambda}{s}}{S}}$ and
$\partial \lr{\close{\frac{ \mu}{s}}{S}}$
are representable as functions $f,g:[a,b]\rightarrow \R$, such that
\begin{align*}
\frac{1}{2}
\Fsl\lr{\close{\frac{\lambda}{s}}{S}}&=
\int_{[a,b]}\sqrt{1+(f')^2} +\frac{s}{\lambda}\lr{ \int_{[a,b]} f-
\frac{1}{2}\abs{S}} \\
\frac{1}{2}
\Fsl\lr{\close{\frac{ \mu}{s}}{S}}&=
\int_{[a,b]}\sqrt{1+(g')^2} +\frac{s}{\lambda}\lr{ \int_{[a,b]} g-
\frac{1}{2}\abs{S}}
\end{align*}

Let $P,Q$ be the points in the positive $y$-plane where
$\partial \close{\frac{\lambda}{s}}{S}$ intersects with $S_1$ and $S_2$.
Define $h:=[a',b']\rightarrow \R$ as the affine function from $P=(a',f(a'))$ to
$Q=(b',f(b'))$ (see Figure
\ref{fig:geomconf}).
Set $\tilde f:=[a',b']\rightarrow \R, \tilde f=h-f$ and
$\tilde g:=[a',b']\rightarrow \R, \tilde g=h-g$, then
$\tilde g(a')=\tilde f(a')=0$,
$\tilde g(b')=\tilde f(b')=0$,
${g'}^2 =({\tilde g}')^2$ and ${f'}^2 =({\tilde f}')^2$.
Note that $\tilde f$ is an arc of circle with radius $\frac{\lambda}{s}$.

\begin{figure}
\begin{center}
\includegraphics[width=14cm]{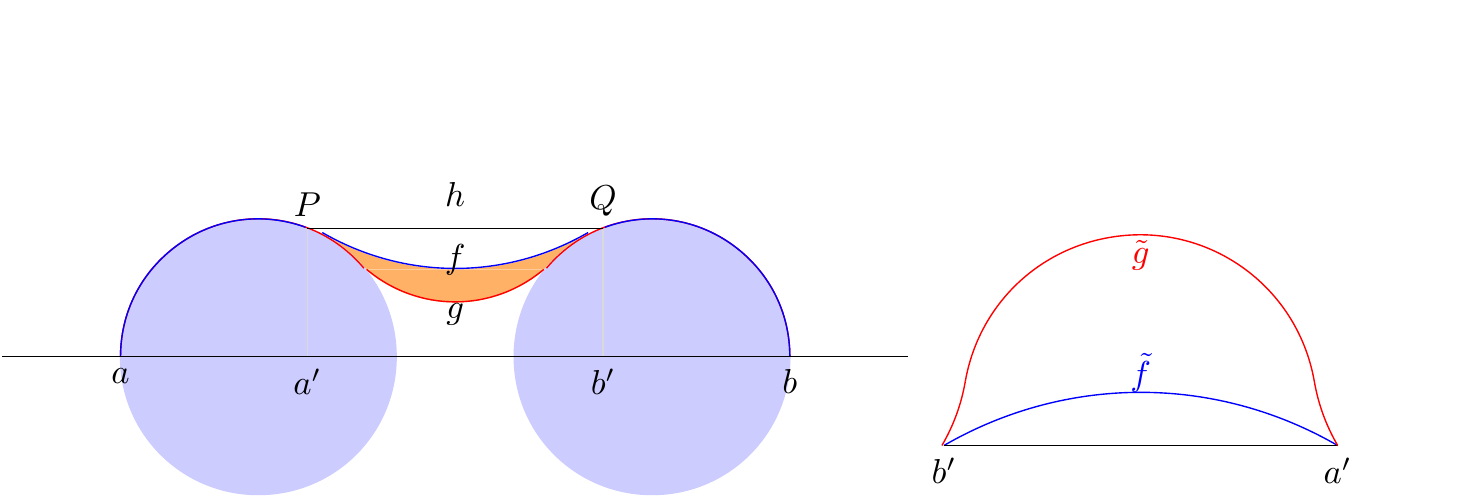}
\end{center}
\caption{The geometric configuration of the proof of Lemma
\ref{le:closingisbest}.
}\label{fig:geomconf}
\end{figure}

Using functional $G_r$ as in \eqref{eq:Gr}  with $r=\frac{\lambda}{s}$,
replacing the domain of integration
$[-R,R]$  by $[a',b']$,
and Lemma \ref{le:arcisbest} (see also  \cite[Lemma 4.29]{ACMBook}) we have
$G_{\frac{\lambda}{s}}(\tilde f) \leq G_{\frac{\lambda}{s}}(\tilde g)$.
Hence
\begin{align*}
\frac{1}{2}
\Fsl\lr{\close{\frac{\lambda}{s}}{S}} -
\frac{1}{2}
\Fsl\lr{\close{\frac{\mu}{s}}{S}} =
G_{\frac{\lambda}{s}}\lr{\tilde f} - G_{\frac{\lambda}{s}}\lr{\tilde g}\leq 0.
\end{align*}
\end{proof}

Next we show some properties of the function
$(s,\lambda)\rightarrow \Fsl\lr{\close{\raussen}{S}}$.

\begin{lemma} \label{le:FSL}
The function $(s,\lambda)\rightarrow \Fsl\lr{\close{\raussen}{S}}$
satisfies the following properties:
\begin{enumerate}[(i)]
\item \label{it:FSL2}
Let $\kappa>0$ and $s_\kappa(\lambda):=1-\frac{\lambda}{\kappa}$, 
$\lambda>0$.
Then the mapping $\lambda \rightarrow
\mathcal{F}_{s_\kappa(\lambda),\lambda}\lr{
\close{\frac{\lambda}{s_\kappa(\lambda)}}{S}}$
is strictly decreasing and continuous as long as
$\frac{\lambda}{s_\kappa(\lambda)}\geq R_c(S)$, i.e.,
as long as the set $\close{\frac{\lambda}{s_\kappa(\lambda)}}{S}$ is connected.

\item  \label{it:FSL3} For $\lambda>0$, the mapping
$s\rightarrow \Fsl\lr{\close{\raussen}{S}}$ is  continuous  and  strictly
increasing for $s \in [0,\frac{\lambda}{R_c(S)}]$, i.e.
as long as the set $\close{\frac{\lambda}{s}}{S}$ is connected.

\item  \label{it:FSL3b} For $r \in  [R_c(S), \infty]$ the functions
 \begin{align*}
 r &\rightarrow P(\close{r}{S}) +\frac{1}{r}\abs{\close{r}{S}}\\
 r&\rightarrow P(\close{r}{S})+\frac{1}{r}\abs{\close{r}{S}\setminus S}
 \end{align*}
are continuous and strictly decreasing in $r$.
\end{enumerate}
\end{lemma}

\begin{proof}
\begin{enumerate}
 \item[(i)] 
Let $\lambda_1 < \lambda_2$ and set $s_i:=1-\frac{\lambda_i }{\kappa}, i=1,2$.
Then $s_2 < s_1$ and $\frac{\lambda_1}{s_1}<\frac{\lambda_2}{s_2}$.
Assume that $\frac{\lambda_1}{s_1} \geq R_c(S)$. By Lemma \ref{le:closingisbest}
we have that
\begin{align*}
\Fcal_{s_2,\lambda_2}\lr{\close{ \frac{\lambda_2}{s_2} }{S}}
&\leq \Fcal_{s_2,\lambda_2}\lr{\close{ \frac{\lambda_1}{s_1} }{S}}\\
&= P\lr{   \close{ \frac{\lambda_1}{s_1} }{S}  } - \frac{1}{\kappa} \abs{\close{
\frac{\lambda_1}{s_1} }{S} \cap S}
+\frac{s_2}{\lambda_2} \abs{\close{ \frac{\lambda_1}{s_1} }{S} \setminus S} \\
&<\Fcal_{s_1,\lambda_1}\lr{\close{ \frac{\lambda_1}{s_1} }{S}}\;.
\end{align*}
Hence the mapping  $\lambda\rightarrow
\mathcal{F}_{s_\kappa(\lambda),\lambda}\lr{
\close{\frac{\lambda}{s_\kappa(\lambda)} }{S}}$ is strictly decreasing
as long as $\frac{\lambda}{s_\kappa(\lambda)}\geq R_c(S)$.
The continuity follows from the continuity of the involved functions.

\item[(ii)--(iii)] Let $\lambda > 0$. The continuity of the function $s\in [0,1]
\rightarrow \Fsl\lr{\close{\raussen}{S}}$
follows from the continuity of the involved functions.
Assume $0\leq s_1<s_2\leq 1$ and $\lambda$ such that
$\close{\frac{\lambda}{s_1}}{S},\close{\frac{\lambda}{s_2}}{S}$ are connected.
Since $R_c(S) \leq \frac{\lambda}{s_2} < \frac{\lambda}{s_1}$,
Lemma \ref{le:closingisbest} gives
\begin{equation}\label{kiko1}
\Fcal_{s_1,\lambda}\lr{\close{ \frac{\lambda}{s_1} }{S}} \leq
\Fcal_{s_1,\lambda}\lr{\close{ \frac{\lambda}{s_2} }{S}},
\end{equation}
hence
\begin{eqnarray*}
\Fcal_{s_1,\lambda}\lr{\close{ \frac{\lambda}{s_1} }{S}} &= & P\lr{\close{
\frac{\lambda}{s_1} }{S}}
+ \frac{s_1}{\lambda}\abs{\close{ \frac{\lambda}{s_1} }{S}}
- \frac{1}{\lambda}\abs{S} \\
& \leq &
P\lr{\close{ \frac{\lambda}{s_2} }{S}} + \frac{s_1}{\lambda}\abs{\close{
\frac{\lambda}{s_2} }{S}}
- \frac{1}{\lambda}\abs{S} \\
& < & P\lr{\close{ \frac{\lambda}{s_2} }{S}}
+ \frac{s_2}{\lambda}\abs{\close{ \frac{\lambda}{s_2} }{S}}  -
\frac{1}{\lambda}\abs{S}
\\
&=& \Fcal_{s_2,\lambda}\lr{\close{ \frac{\lambda}{s_2} }{S}}.
\end{eqnarray*}
This proves
that the mapping $s\rightarrow \Fsl\lr{\close{\raussen}{S}}$ is strictly
increasing on the interval
$\left[ 0,\min\set{ 1,\frac{\lambda}{R_c(S)} } \right]$.
Adding $\frac{1}{\lambda}\abs{S}$ to both sides of the inequality
and setting $r_1:=\frac{\lambda}{s_1}, r_2:=\frac{\lambda}{s_2}$
we have that $0 < r_2 < r_1$ and
$$P\lr{\close{r_1}{S}}+\frac{1}{r_1}\abs{\close{r_1}{S}}
< P\lr{\close{r_2}{S}}+\frac{1}{r_2}\abs{\close{r_2}{S}}.$$
Hence $P\lr{\close{r}{S}}+\frac{1}{r}\abs{\close{r}{S}}$ is strictly decreasing
in $r$ for $r \in  [R_c(S), \infty]$
if $R_c(S) > 0$, or in $(R_c(S), \infty]$ if $R_c(S) = 0$. We understand that
the function is $+\infty$ when $r=R_c(S) = 0$.

Writing again \eqref{kiko1} as
\begin{eqnarray*}
& & P\lr{\close{ \frac{\lambda}{s_1} }{S}}
+ \frac{s_1}{\lambda}\abs{\close{ \frac{\lambda}{s_1} }{S} \setminus S}
- \frac{1-s_1}{\lambda}\abs{S} \\
& \leq &
P\lr{\close{ \frac{\lambda}{s_2} }{S}} + \frac{s_1}{\lambda}\abs{\close{
\frac{\lambda}{s_2} }{S}\setminus S}
- \frac{1-s_1}{\lambda}\abs{S} \\
& < & P\lr{\close{ \frac{\lambda}{s_2} }{S}} + \frac{s_2}{\lambda}\abs{\close{
\frac{\lambda}{s_2} }{S}\setminus S}
- \frac{1-s_1}{\lambda}\abs{S},
\end{eqnarray*}
setting $r_1:=\frac{\lambda}{s_1}, r_2:=\frac{\lambda}{s_2}$ (notice that these
values are not $0$),
and adding $\frac{1-s_1}{\lambda}\abs{S}$ to both sides of the inequality above,
 we get that
$P\lr{\close{r}{S}}+\frac{1}{r}\abs{\close{r}{S}\setminus S}$ is strictly 
decreasing in $r$ for
$r \in  [R_c(S), \infty]$ if $R_c(S) > 0$, or in $(R_c(S), \infty]$ if $R_c(S) =
0$. Notice that the function is also continuous in that range.

It remains to consider the case $R_c(S) = 0$
and to prove that $$P\lr{\close{r}{S}}+\frac{1}{r}\abs{\close{r}{S}\setminus
S}$$ is continuous at $r=0$.
This follows if we prove that
\begin{equation}\label{TS3}
\frac{1}{r}\abs{\close{r}{S}\setminus S}\to 0+
 \qquad \hbox{\rm as $r\to 0+$.}
\end{equation}
To prove \eqref{TS3}, let us estimate the area of $\close{r}{S}\setminus S$.
This set is contained in a triangle
whose basis has length $\leq 2r$ and whose height is less than  $\sqrt{r}
\max(\sqrt{2R_1-r},\sqrt{2R_2-r}) = O(\sqrt{r})$.
Hence $\abs{\close{r}{S}\setminus S}=O(r^{3/2})$ and \eqref{TS3} holds.
\end{enumerate}
\end{proof}

\begin{lemma}\label{le:noTransversalSet}
Assume that \eqref{contextenodestruccio} holds and that
 $\frac{\abs{S_1}}{P(S_1)}<\frac{\abs{S}}{P(co(S))}$. Then
 $\Cs \cap S_2 \in \set{\emptyset,S_2}$, i.e., 
 $\Cs$ cannot be a transversal set.
\end{lemma}

\begin{proof}
From \eqref{eqmax} and Proposition \ref{pro:gnorm} we obtain that 
\begin{equation}\label{l3}
\lambda_3:=\frac{\abs{S}}{P(co(S))}
= \max_{X \subset \R^2}\rho(X)\,.
\end{equation}

We claim that $\lambda_3\le r_2$. Indeed, if $\lambda_3>r_2$
then by Lemma \ref{le:arcisbest} we have that the set $co(S)$
cannot be a minimizer of $\Fcal_{0,\lambda_3}$, that is,
there exists $X\subset\R^2$ such that 
$$
\Fcal_{0,\lambda_3}(X)<\Fcal_{0,\lambda_3}(co(S))=0\,.
$$
It then follows $\rho(X) > \lambda_3$, contradicting \eqref{l3}.
We thus proved that $\lambda_3\le r_2$.

We now claim that, for $s>1-\frac{\lambda}{\lambda_3}$
we necessarily have $C_{s,\lambda}=\emptyset$. 
Indeed, it is enough to show that the empty set is a minimizer 
of $\Fsl$ for $s=1-\frac{\lambda}{\lambda_3}$, that is,
$\Fsl(C_{s,\lambda})=0$. Since 
$\frac{1-s}{\lambda}=\frac{1}{\lambda_3}\ge\frac{1}{r_2}$,
from Proposition \ref{lemma:gg} it follows that 
$C_{s,\lambda}\in 
\left\{S_1,S,\emptyset,\close{\frac{s}{\lambda}}{S}\right\}$.
From our assumptions  it directly follows that 
$\Fsl(S_1)>0$ and $\Fsl(S)>0$, hence it remains to show that 
$\Fsl(\close{\frac{s}{\lambda}}{S})\ge 0$.
By Lemma \ref{le:FSL} (i) we know that 
$$
\Fsl(\close{\frac{s}{\lambda}}{S})=
\Fcal_{1-\frac{\lambda}{\lambda_3},\lambda}
(Close_{\frac{1}{\lambda}-\frac{1}{\lambda_3}}(S))
\ge \Fcal_{0,\lambda_3}(co(S))=0\,,
$$
which proves our claim.

In particular, if $C_{s,\lambda}\ne\emptyset$ it follows that 
$$
s\le 1-\frac{\lambda}{\lambda_3}\le 1-\frac{\lambda}{r_2}\,,
$$
i.e., $\frac{1-s}{\lambda}\ge\frac{1}{r_2}$,
hence $C_{s,\lambda}$ cannot be a transversal set (again by 
Proposition \ref{lemma:gg} \ref{it:gg2}).
\end{proof}

\smallskip

{\it Proof of Proposition \ref{de:lambda}}.\\

\begin{enumerate}[(a)]
\item  Assume first that  $R_c(S) > 0$. In that case, because of the convexity
of $S_1$, $S_2$, $P(S) < P(\close{R_c(S)}{S})$, and we have
\begin{equation*} P(S)<  P\lr{\close{R_c(S)}{S}}+\frac{1}{R_c(S)}
 \abs{ \close{ R_c(S)}{S}\setminus S}.
 \end{equation*}
 On the other hand
 \begin{equation}\label{TS1}
\lim_{r\rightarrow \infty}  P(\close{r}{S})
+\frac{1}{r}\abs{\close{r}{S}\setminus S} =P(co(S)) < P(S)\;.
\end{equation}
Hence, by Lemma \ref{le:FSL}.(iii), \eqref{eq:R1} has a unique solution in
$R_1\in (R_c(S),\infty)$.

Assume now that  $R_c(S) = 0$. Then we have $\close{R_c(S)}{S} = S$, moreover,
since
$f(r):= P(\close{r}{S}) +\frac{1}{r}\abs{\close{r}{S}\setminus S}$
is a continuous and decreasing function in $[R_c(S),\infty)$ by Lemma
\ref{le:FSL}.(iii), we have
\begin{equation*}
 \lim_{r\to 0+}P(\close{r}{S}) +\frac{1}{r}\abs{\close{r}{S}\setminus S} \geq
\lim_{r\to 0+}P(\close{r}{S}) = P(S).
\end{equation*}
On the other hand, we also have \eqref{TS1}. Thus, 
$R_1=0 \in [0,\infty)$ satisfies \eqref{eq:R1}. This value is unique since 
$r\rightarrow P(\close{r}{S})+ \frac{1}{r}\abs{\close{r}{S}\setminus S}$ is
strictly decreasing.

To prove \eqref{eq:lambda1Circles}, let us observe that
$$
\Fcal_{s,\lambda}\lr{\close{\frac{\lambda}{s}}{S}} - \Fcal_{s,\lambda}\lr{S}=
P\lr{\close{ \frac{\lambda}{s} }{S}} + \frac{s}{\lambda}\abs{\close{
\frac{\lambda}{s} }{S}\setminus S}
- P(S).
$$
Setting $s=s_2(\lambda)$ and $r=\frac{\lambda}{s_2(\lambda)}$ in the above
equality and observing that
$\frac{\lambda}{s_2(\lambda)}$ is an increasing function of $\lambda$ we have
that
$$
\Fcal_{s_2(\lambda),\lambda}\lr{\close{\frac{\lambda}{s_2(\lambda)}}{S}} -
\Fcal_{s_2(\lambda),\lambda}\lr{S}
$$
is $> 0$  (resp. $=0$, $<0$) if and only if $\lambda < \lambda_1$ (resp.
$\lambda = \lambda_1$, $\lambda > \lambda_1$).

Notice that $\lambda_1= 0$ if and only if $R_1=0$ and we have proved that this
happens if and only if $R_c(S)=0$.

\item 
\begin{enumerate}[(i)]
\item  We are assuming that $P(co(S)) < \frac{P(S_1)}{|S_1|}|S|$.
Let us first assume that the radius of $S_1$ is $>$ than the radius of $S_2$,
hence
$\frac{P(S)}{|S|} > \frac{P(S_1)}{|S_1|}$. If $R_c(S) > 0$, then
\begin{equation*}
\begin{array}{ll}
P\lr{\close{r}{S}}+ \frac{1}{r}\abs{\close{r}{S}\setminus S}\to
P\lr{\close{R_c(S)}{S}}+ \frac{1}{R_c(S)}\abs{\close{R_c(S)}{S}\setminus S} \\
\\
\qquad \qquad \qquad \qquad \qquad \qquad \qquad \qquad  > P(S) >
\frac{P(S_1)}{|S_1|}|S| \qquad \hbox{\rm as $r\to R_c(S)$}
\end{array}
\end{equation*}
and
\begin{equation*}
P\lr{\close{r}{S}}+ \frac{1}{r}\abs{\close{r}{S}\setminus S}\to P(S) >
\frac{P(S_1)}{|S_1|}|S|
 \qquad \hbox{\rm as $r\to 0+$}
\end{equation*}
in case that $R_c(S) = 0$.
On the other hand
\begin{equation*}
P\lr{\close{r}{S}}+ \frac{1}{r}\abs{\close{r}{S}\setminus S}\to P(co(S)) <
\frac{P(S_1)}{|S_1|}|S|.
 \qquad \hbox{\rm as $r\to \infty$}\;.
\end{equation*}
Thus there is a unique value $R_2 \in (R_c(S),\infty)$ satisfying 
\eqref{eq:R2}.

Now, if  $\frac{P(S)}{|S|} = \frac{P(S_1)}{|S_1|}$, then
both equations \eqref{eq:R1} and \eqref{eq:R2} are the same and we can take
$R_2=R_1$.
Hence $\lambda_2=\lambda_1$.
Clearly, if $R_2=R_1$, then $\frac{P(S)}{|S|} = \frac{P(S_1)}{|S_1|}$.
Notice that if  $\lambda_1 = \lambda_2$, then
$$
0 \geq (R_1-R_2) |S_1||S_2| = R_1R_2 (|S_1|P(S_2)-|S_2|P(S_1))\geq 0.
$$
Thus $R_1=R_2$. Note that if
$\frac{P(S)}{|S|} = \frac{P(S_1)}{|S_1|}$ and $R_c(S) = 0$,  by $(i)$, $R_2 =
R_1=0$ and $\lambda_2=\lambda_1=0$.

To prove \eqref{eq:lambda2Circles} we proceed as in the proof of $(i)$. The fact
that $\lambda_2\geq \lambda_1$
follows since $\frac{P(S_1)}{|S_1|}\leq \frac{P(S)}{|S|}$. From the explicit
formula for $\lambda_2$, it follows that
$\lambda_2 \leq \frac{P(S_1)}{|S_1|}$.

\item 
 For $\lambda \in \left(\fraczwei,\fraceins\right)$ the only possible minimizers
for $\Fcal_{0,\lambda}$ are $S_1$ and $\Gamma_{0,\lambda}$. For
$\lambda = \fraczwei$, $\Gamma_{0,\lambda}$ is a minimizer.

Proposition \ref{pro:gnorm} states that for 
$\lambda>\fraceins$, $\emptyset$ is the only possible minimizer. 
Since $\lambda \rightarrow \min_X \Fcal_{0,\lambda}\lr{X}$ is continuous, 
there has to be $\lambda=\lambda_2 \in \lr{\fraczwei,\fraceins}$ such that 
$\Fcal_{0,\lambda_2}\lr{S_1} =
\Fcal_{0,\lambda_2}\lr{\Gamma_{0,\lambda_2}\lr{S}}$. 
Rearrangement of this equation gives
\begin{equation*}
\lambda_2 = \frac{\abs{\Gamma_{0,\lambda_2}\lr{S}\cap S_2}}{
P\lr{\Gamma_{0,\lambda_2}\lr{S}}-P\lr{S_1}}\;.
\end{equation*}

\end{enumerate}
\end{enumerate}

\qed

\smallskip

\section{Proof of Theorem \ref{pr:explicitCircles}}\label{sect:proof}
We consider the three different intervals of $\lambda$. For each of them we
compute $C_{s,\lambda}$ for $s\in [0,1]$.

\begin{enumerate}[(a)]
\item  $\lambda\in [0,\lambda_1]$. In this case $\lambda<r_2$ such that
$\Gamma_{s,\lambda}\lr{S} = \close{\raussen}{S}$.

\begin{enumerate}[{(a}1{)}]
\item 
Let us prove that there is a function $s_a(\lambda)$, $\lambda\in
[0,\lambda_1]$, such that
\begin{equation}\label{eq:condA}
\Fcal_{s,\lambda}\lr{\close{\frac{\lambda}{s}}{S}} <
\Fcal_{s,\lambda} \lr{S}\; \qquad \hbox{\rm (respectively $=$, $>$)}
\end{equation}
if and only if $s \in [0,s_a(\lambda))$, resp. $s=s_a(\lambda)$, $s >
s_a(\lambda)$.
Notice that $\Fcal_{s,\lambda}\lr{\close{\frac{\lambda}{s}}{S}} \leq
\Fcal_{s,\lambda} \lr{S}$
if and only if
\begin{equation}\label{eq:garrafa1x}
P\lr{\close{\frac{\lambda}{s}}{S}} +\frac{s }{\lambda}
\abs{\close{\frac{\lambda}{s }}{S} \setminus S} \leq P(S)\;.
\end{equation}
Clearly, by Proposition \ref{de:lambda} \ref{it:lambdaa}, if we define
\begin{equation*}
s_a(\lambda)=\frac{1}{R_1}\lambda \quad \lambda \in [0,  \lambda_1]\;,
\end{equation*}
then the equality in (\ref{eq:garrafa1x}) holds identically.
Now, by Lemma \ref{le:FSL} \ref{it:FSL3b}, the left hand side of
(\ref{eq:garrafa1x}) is an
increasing function of $s$,
and the identity in (\ref{eq:garrafa1x}) only holds at $s=s_a(\lambda)$. Thus
(\ref{eq:condA}) holds.

Remark that (\ref{eq:condA}) holds for any value of $\lambda$.

\item 
{\it Identification of $C_{s,\lambda}$.}
Recall that, by Lemma \ref{le:comparScircle}, for any $s\in \left[0,1-\lambda 
\fraczweii\right]$ we have
$$
\min\set{
\Fcal_{s,\lambda} \lr{S},
\Fcal_{s,\lambda} \lr{S_1},
\Fcal_{s,\lambda} \lr{S_2},
0
}=\Fcal_{s,\lambda} \lr{S}.
$$
Thus
$C_{s,\lambda}= \close{\frac{\lambda}{s}}{S}$ if $s \in [0,s_a(\lambda)]$,
$C_{s,\lambda}=S$ if $s\in \left(s_a(\lambda),1-\lambda  \fraczweii \right]$.
Notice that if $s=s_a(\lambda)$, $S$ is also a minimizer of
$\Fcal_{s,\lambda}$.

Using (\ref{eq:condA}) and Lemma \ref{le:comparScircle} we clearly have that
$C_{s,\lambda}=S_1$ if $s\in \left(1-\lambda  \fraczweii,1-\lambda\fraceinsi
\right]$ and
$C_{s,\lambda}=\emptyset$ if $s > 1-\lambda\fraceinsi$.

\end{enumerate}
\item  Let $\lambda \in (\lambda_1,\lambda_2]$.
In this case,  let us prove that there is a function $s_b(\lambda)$ such that
\begin{equation}\label{eq:S1}
\Fcal_{s_b(\lambda),\lambda}\lr{\Gamma_{s_b(\lambda),\lambda}\lr{S}} =
\Fcal_{s_b(\lambda),\lambda} \lr{S_1}
\qquad \lambda \in
[\lambda_1,\lambda_2]\;.
\end{equation}
Let us consider two cases  $\frac{|S_1|}{P(S_1)} <
\frac{|S|}{P(\mathrm{co}(S))}$ and 
 $\frac{|S_1|}{P(S_1)} \geq
\frac{|S|}{P(\mathrm{co}(S))}$.

\begin{enumerate}[]
\item[(b1)] \label{st:b1} In this case we  assume that 
$\frac{\abs{S_1}}{P(S_1)} < 
  \frac{\abs{S}}{P(\mathrm{co}(S))}$.
\begin{enumerate}[i)]
\item 
{\it Proof of \eqref{eq:S1}.} 
Recall from Lemma \ref{le:noTransversalSet} that in this situation we have 
$\Gamma_{s,\lambda}\lr{S} = \close{\raussen}{S}$.
In this case $\lambda_2 = \frac{R_2|S_1|}{R_2P(S_1)+|S_1|}$. We have
$$
s_a(\lambda) = \frac{\lambda}{R_1} \leq 1-\lambda \frac{P(S_1)}{|S_1|}.
$$
if and only if
$$
\lambda   \leq \frac{R_1|S_1|}{R_1P(S_1) + |S_1|}:=\bar   \lambda_1.
$$
Observe that $\bar   \lambda_1 \leq \lambda_2$ since $R_1\leq R_2$.

Let us work in the interval $s\in \left[
0,\frac{\lambda}{R_1} \right]$ for all $\lambda \in [\lambda_1,\lambda_2]$.  Let
us prove that
\begin{equation}\label{gamba1}
\Fcal_{0,\lambda}\lr{\mathrm{co}(S)}<  \Fcal_{0,\lambda}\lr{S_1}\;.
\end{equation}
Indeed, since $\frac{|S_1|}{P(S_1)}< \frac{|S|}{P(\mathrm{co}(S))}$, after
some simple computations
we deduce that
$$
\lambda_2 = \frac{R_2|S_1|}{R_2P(S_1) + |S_1|} <
\frac{|S_2|}{P(\mathrm{co}(S)) - P(S_1)}.
$$
Thus, if $\lambda \leq \lambda_2$, then $\lambda <
\frac{|S_2|}{P(\mathrm{co}(S)) - P(S_1)}$ and this is equivalent to
\eqref{gamba1}.

Moreover, by Proposition \ref{de:lambda} \ref{it:lambdab} (assuming that
$\frac{|S_1|}{P(S_1)}<\frac{|S|}{P(\mathrm{co}(S))}$),
for $\lambda < \lambda_2$ and $s=s_1(\lambda)=1-\lambda \frac{P(S_1)}{|S_1|}$ we
have
\begin{equation*}
\Fcal_{ s_1(\lambda),\lambda}\lr{S_1}<
\Fcal_{s_1(\lambda),\lambda}\lr{ \close{\frac{\lambda}{s_1(\lambda)}}{S}},
\end{equation*}
with equality if $\lambda = \lambda_2$,
and for $\lambda \in (\lambda_1,\lambda_2]$
and $s=\frac{\lambda}{R_1}$, we have
\begin{equation}\label{pv1}
 \Fcal_{\frac{\lambda}{R_1} ,\lambda}\lr{S_1}<
\Fcal_{\frac{\lambda}{R_1},\lambda}(S) =
\Fcal_{\frac{\lambda}{R_1},\lambda}\lr{ \close{R_1}{S}}
\end{equation}
(the first inequality being true because $\lambda>\lambda_1$).
Since
both functions $s\rightarrow \Fcal_{s,\lambda}(S_1)$ and $s\rightarrow 
\Fsl\lr{\close{\frac{\lambda}{s}}{S}}$
are continuous in $s$, they have to intersect for some $s \in 
\left[0,\min\set{ \frac{\lambda}{R_1},1-\lambda \frac{P(S_1)}{|S_1|}} \right]$.
Hence there is at least one value $s$ that satisfies \eqref{eq:S1}.
Let $s_b(\lambda)$ be the smallest value of $s$ satisfying \eqref{eq:S1}.

Notice that we have that $s_b(\lambda) < \frac{\lambda}{R_1}$ for any $\lambda
\in (\lambda_1,\lambda_2]$
and $s_b(\lambda) < 1-\lambda \frac{P(S_1)}{|S_1|}$ for any $\lambda <
\lambda_2$ (with equality if $\lambda=\lambda_2$).

\item 
\label{st:b2}
We show that $s_b(\lambda)$ is the unique value of
$s \in \left[0,\min\left\{\frac{\lambda}{R_1},1-\lambda
\frac{P(S_1)}{|S_1|}\right\} \right]$
satisfying \eqref{eq:S1}, if $\lambda \in (\lambda_1,\lambda_2]$.

Clearly, if $\lambda=\lambda_2$, then $s_b(\lambda)=1-\lambda
\frac{P(S_1)}{|S_1|}=\min\left\{\frac{\lambda}{R_1},1-\lambda
\frac{P(S_1)}{|S_1|}\right\} $ and \eqref{eq:S1} holds.
Our assertion is true in this case.

Assume that $\lambda \in (\lambda_1,\lambda_2)$. Let us prove that for any $s\in
\left(s_b(\lambda),\min\left\{\frac{\lambda}{R_1},1-\lambda
\frac{P(S_1)}{|S_1|}\right\}\right]$ we have
\begin{equation}\label{patataX1}
\Fcal_{s,\lambda}(S_1) <  \Fcal_{s,\lambda}(\mathrm{Close}_{\frac{\lambda}{s}}
(S))\;.
\end{equation}
Suppose that we find $s_b(\lambda) < t_1 < t_2 <
\min\left\{\frac{\lambda}{R_1},1-\lambda \frac{P(S_1)}{|S_1|}\right\}$ where
\begin{align}\label{s01}
\Fcal_{t_1,\lambda}(S_1) &< 
\Fcal_{t_1,\lambda}(\mathrm{Close}_{\frac{\lambda}{t_1}} (S))\;,\\
\label{s02}
\Fcal_{t_2,\lambda}(S_1) &> 
\Fcal_{t_2,\lambda}(\mathrm{Close}_{\frac{\lambda}{t_2}} (S))\;.
\end{align}

Let us compute $C_{t_1,\lambda}$. Observe that 
by \eqref{s01}
$S_1$  has less
energy
than $\Gamma_{t_1,\lambda}\lr{S}$, 
and $\Gamma_{t_1,\lambda}\lr{S}$  is better than $S$  because $t_1
< \frac{\lambda}{R_1}$.
Also $\Fcal_{t_1,\lambda}(S_1)< 0$  because $t_1 < 1-\lambda
\frac{P(S_1)}{|S_1|}$.
Thus $\Fcal_{t_1,\lambda}(S) = \Fcal_{t_1,\lambda}(S_1) +
\Fcal_{t_1,\lambda}(S_2) \leq \Fcal_{t_1,\lambda}(S_2)$.
Thus $C_{t_1,\lambda} = S_1$.

Let us compute $C_{t_2,\lambda}$. Observe that
$\Gamma_{t_2,\lambda}\lr{S}$  is better than $S_1$ (by
(\ref{s02})).
And $\Gamma_{t_2,\lambda}\lr{S}$  is better than $S$  because $t_2
< \frac{\lambda}{R_1}$.
Also $\Fcal_{t_2,\lambda}(S_1)\leq 0$  because $t_2 <  1-\lambda
\frac{P(S_1)}{|S_1|}$.
Thus $\Fcal_{t_2,\lambda}(S) = \Fcal_{t_2,\lambda}(S_1) +
\Fcal_{t_2,\lambda}(S_2) \leq \Fcal_{t_2,\lambda}(S_2)$.
Thus $C_{t_2,\lambda}=\Gamma_{t_2,\lambda}\lr{S}$.

It is not possible that $t_1 < t_2$ and the optimal set $C_{t_2,\lambda}$ 
contains the optimal set $C_{t_1,\lambda}$. 
We conclude that $\Fcal_{s,\lambda}(S_1) \leq 
\Fcal_{s,\lambda}\lr{\Gamma_{s,\lambda}\lr{S}}$
for all $s\in \left(s_b(\lambda),\min\left\{\frac{\lambda}{R_1},1-\lambda
\frac{P(S_1)}{|S_1|}\right\}\right)$.
The inequality has to be strict. Otherwise there would be two points
$t < t'$ such that the minima of $\Fcal_{t,\lambda}$ are both $S_1$ and
$\Gamma_{t,\lambda}\lr{S}$
and the minima of $\Fcal_{t',\lambda}$ are both $S_1$ and
$\Gamma_{t',\lambda}\lr{S}$.
Then $S_1$ would contain $\Gamma_{t',\lambda}\lr{S}$, a
contradiction.
Thus \eqref{patataX1} is proved.

\item 
\label{st:b3}
{\it 
Computation of $C_{s,\lambda}$ for $\lambda\in (\lambda_1,\lambda_2]$ 
and any $s$.}

If $s\in [0,s_b(\lambda))$ we argue as for $t_2$ and we deduce that the optimum
is
$\mathrm{Close}_{\frac{\lambda}{s}} (S)$. Letting $s \to s_b(\lambda)$ we deduce
that
$C_{s_b(\lambda),\lambda} = \mathrm{Close}_{\frac{\lambda}{s_b(\lambda)}} (S)$.

If $s\in \left( s_b(\lambda),\min \set{ \frac{\lambda}{R_1},1-\lambda
\frac{P(S_1)}{|S_1|} } \right]$ 
we argue as
for $t_1$ and we deduce that the optimum is $S_1$.

If $\min\left\{ \frac{\lambda}{R_1},1-\lambda
\frac{P(S_1)}{|S_1|}\right\}=\frac{\lambda}{R_1}$, i.e., if $\lambda\in
(\lambda_1,\bar\lambda_1]$,
let us compute the optimum for $s\in (\frac{\lambda}{R_1},1-\lambda
\frac{P(S_1)}{|S_1|}]$.
On this interval $\Fcal_{s,\lambda}(S_1)\leq 0$.
Thus $\Fcal_{s,\lambda}(S) = \Fcal_{s,\lambda}(S_1) + \Fcal_{s,\lambda}(S_2)
\leq \Fcal_{s,\lambda}(S_2)$.
Also (by the definition of $R_1$) on this interval
$\Fcal_{s,\lambda}(\mathrm{Close}_{\frac{\lambda}{s}} (S)) >
\Fcal_{s,\lambda}(S)$. Thus the optimum is either $S_1$ or $S$.
By monotonicity of the optimum with respect to $s$ and the fact that the optimum
in
 $(s_b(\lambda),\frac{\lambda}{R_1}]$  is $S_1$, we deduce that it is also $S_1$
 in $(\frac{\lambda}{R_1},1-\lambda \frac{P(S_1)}{|S_1|}]$.

If $\min\left\{\frac{\lambda}{R_1},1-\lambda
\frac{P(S_1)}{|S_1|}\right\}=\frac{\lambda}{R_1}$ and $s > 1-\lambda
\frac{P(S_1)}{|S_1|}$,
we have that $\Fcal_{s,\lambda}(S_1) > 0$ and the minimum is $\emptyset$.

If $\min\left\{ \frac{\lambda}{R_1},1-\lambda \frac{P(S_1)}{|S_1|}
\right\}=1-\lambda \frac{P(S_1)}{|S_1|}$,
i.e., if $\lambda\in (\bar\lambda_1,\lambda_2]$,
as in the previous paragraph the minimum for $s >1-\lambda \frac{P(S_1)}{|S_1|}$
is $\emptyset$.

Let us point out that for $\lambda=\lambda_2$ and for $s > s_b(\lambda_2)=
1-\lambda_2 \frac{P(S_1)}{|S_1|}$, we have that
$$
\Fcal_{s,\lambda_2}(\mathrm{Close}_{\frac{\lambda_2}{s}} (S)) >
\Fcal_{s_b(\lambda_2),\lambda_2}(\mathrm{Close}_{\frac{\lambda_2}{s_b(\lambda_2)
}} (S))
= \Fcal_{s_b(\lambda_2),\lambda_2}(S_1) = 0.
$$
Since also $\Fcal_{s,\lambda_2}(S_1) >0$, then $C_{s,\lambda_2} =\emptyset$.
\end{enumerate}

\item[(b2)] Assume that $\frac{|S_1|}{P(S_1)}\geq
\frac{|S|}{P\lr{\mathrm{co}(S) }}$.
\begin{enumerate}[i)]
\item 
{\it  
 Define $s_b(\lambda)$ for  $\lambda \in
(\lambda_1,\lambda_2]$.}
In this case $\lambda_2 = \frac{|S_2|}{P(\mathrm{co}(S))-P(S_1)}$.

Since $\lambda\leq \lambda_2 = \frac{|S_2|}{P(\mathrm{co}(S)-P(S_1)}$, then
\eqref{gamba1} holds.
On the other hand \eqref{pv1} also holds for any $\lambda \in
(\lambda_1,\lambda_2]$.
As in paragraph (b1) we identify a solution
$s_b(\lambda)\in \left[0,\frac{\lambda}{R_1} \right]$ of (\ref{eq:S1}). Let us
take the smallest one.
Let us prove that $s_b(\lambda) \in \left[0,1-\lambda \frac{P(S_1)}{|S_1|}
\right]$. Let
$s=s_1(\lambda)= 1-\lambda \frac{P(S_1)}{|S_1|}$. Then
$\Fcal_{s_1(\lambda),\lambda}(S_1) = 0$ and
$$
\Fcal_{s_1(\lambda),\lambda}(\mathrm{Close}_{\frac{\lambda}{s_1(\lambda)}} (S))=
P(\mathrm{Close}_{\frac{\lambda}{s_1(\lambda)}} (S)) +
\frac{s_1(\lambda)}{\lambda} |\mathrm{Close}_{\frac{\lambda}{s_1(\lambda)}}
(S)\setminus S| - \frac{P(S_1)}{|S_1|}|S|
$$
$$
\geq P(\mathrm{co}(S))- \frac{P(S_1)}{|S_1|}|S| \;  \geq \; 0 \;.
$$
Notice that the second inequality is strict if
$\frac{|S_1|}{P(S_1)}> \frac{|S|}{P(\mathrm{co}(S))}$, while the first is strict
if
$\frac{|S_1|}{P(S_1)} = \frac{|S|}{P(\mathrm{co}(S))}$ and $s_1(\lambda)> 0$.
In both cases, since $s_b(\lambda)$ is the smallest solution of (\ref{eq:S1}),
we have $s_b(\lambda) \in [0,1-\lambda \frac{P(S_1)}{|S_1|}]$.

If $\frac{|S_1|}{P(S_1)} = \frac{|S|}{P(\mathrm{co}(S))}$,
and $s_1(\lambda)= 0$, then $\lambda= \frac{|S_1|}{P(S_1)}=\lambda_2$ and
$$
\Fcal_{s_1(\lambda_2),\lambda_2}\lr{
\Gamma_{s_1(\lambda_2),\lambda_2}\lr{S}
} =
\Fcal_{s_1(\lambda_2),\lambda_2}(S_1) = 0,
$$
and we take $s_b(\lambda_2)= 0$.

\item {\it  Computation
of $C_{s,\lambda}$ for  $\lambda \in (\lambda_1,\lambda_2]$ and any $s$.}

If $\lambda \in (\lambda_1,\bar\lambda_1]$, then $s_b(\lambda) \leq
\frac{\lambda}{R_1} \leq 1-\lambda \frac{P(S_1)}{|S_1|}$.
Then the argument is identical to the same case in Step \ref{st:b3}.

If $\lambda \in (\bar \lambda_1,\lambda_2]$, then $ 1-\lambda
\frac{P(S_1)}{|S_1|}< \frac{\lambda}{R_1}$.
We have that 
$
\Fcal_{s,\lambda}(\Gamma_{s,\lambda}\lr{S})
\leq
\Fcal_{s,\lambda}(S)$ 
for all 
$s\in \left[0,\frac{\lambda}{R_1}\right]$.
If $s \in \left[ 0,1-\lambda \frac{P(S_1)}{|S_1|} \right]$ we have
$\Fcal_{s,\lambda}(S_1)\leq 0$.
Then $\Fcal_{s,\lambda}(S)=\Fcal_{s,\lambda}(S_1)+\Fcal_{s,\lambda}(S_2)\leq
\Fcal_{s,\lambda}(S_2)$.
Thus the optimal set for $s \in \left[0,1-\lambda \frac{P(S_1)}{|S_1|}\right]$ 
can
be only 
$\GammaS$ 
or $S_1$.

As in step \ref{st:b2} we prove that $s_b(\lambda)$ is the unique value of
$s \in \left[0,\frac{\lambda}{R_1} \right]$
satisfying \eqref{eq:S1} if $\lambda \in (\lambda_1,\lambda_2]$.
Then we proceed as in Step \ref{st:b3} to prove that
$C_{s,\lambda}= \GammaS$ 
if $s\in
[0,s_b(\lambda)]$,
and $C_{s,\lambda}=S_1$ if $s\in
\left(s_b(\lambda),1-\lambda\frac{P(S_1)}{|S_1|} \right]$.
For $s> 1-\lambda\frac{P(S_1)}{|S_1|}$, $C_{s,\lambda}=\emptyset$ since
$\Fcal_{s,\lambda}(S_1) > 0$.
\end{enumerate}
\end{enumerate}

\item  Let $\lambda\in (\lambda_2,\lambda_3]$. Again we distinguish two
cases $\fraceins < \fracco$ and $\fraceins \geq \fracco$.

\begin{enumerate}[]
\item[(c1)] 
{\it 
Assume that $\fraceins < \fracco$.}
\begin{enumerate}[i)]
\item Recall from Lemma \ref{le:noTransversalSet} we have again that
$\GammaS = \close{\raussen}{S}$.
In this case $\lambda_2=\frac{R_2|S_1|}{R_2P(S_1)+|S_1|}$  and $\lambda_3 =
\fracco$.
We look for a value $s_c(\lambda)$, $\lambda \in (\lambda_2,\lambda_3]$, such
that
\begin{equation} \label{eq:S2}
\Fcal_{s_c(\lambda),\lambda}\lr{\close{\frac{\lambda}{s_c(\lambda)}}{S}} = 0.
\end{equation}

Let $\lambda \in (\lambda_2,\lambda_3)$. If $s=0$, we have
$$\Fcal_{0,\lambda}\lr{co(S)}< 0.$$
Let $s_3(\lambda)=1-\lambda \frac{P(co(S))}{\abs{S}}$. Since by Lemma
\ref{le:FSL}
$\lambda\rightarrow
\mathcal{F}_{s_3(\lambda),\lambda}\lr{\close{\frac{\lambda}{s_3(\lambda)}}{S}}$ 
is strictly decreasing, then
\begin{equation*}
0=\Fcal_{0,\fracco}\lr{co(S)}<
\mathcal{F}_{s_3(\lambda),\lambda}\lr{\close{\frac{\lambda}{s_3(\lambda)}}{S}}.
\end{equation*}
Then for any $\lambda \in (\lambda_2,\lambda_3)$ there exists $s_c(\lambda) <
s_3(\lambda)$ such that
\eqref{eq:S2} holds.

Let $\lambda=\lambda_3$. Then $\mathcal{F}_{0,\lambda_3}(co(S))=0$ and
$s_3(\lambda_3) = 0$. Hence
$\mathcal{F}_{s_3(\lambda),\lambda}\lr{\close{\frac{\lambda}{s_3(\lambda)}}{S}}
=0$.
Since $\mathcal{F}_{s,\lambda_3}\lr{\close{\frac{\lambda_3}{s}}{S}}$ is strictly
increasing in $s$
we have that $\mathcal{F}_{s,\lambda_3}\lr{\close{\frac{\lambda_3}{s}}{S}} > 0$
for any 
$s\in \left(0,\frac{\lambda_3}{R_c(S)} \right]$.
Then $s_c(\lambda_3) = 0 = s_3(\lambda_3)$.

\item 
Let
us prove that
\begin{equation} \label{botifarraX2}
\hbox{\rm if $\lambda\geq \lambda_2 = \frac{R_2|S_1|}{R_2P(S_1)+|S_1|}$ and 
$s\in  \left[0,1-\lambda\frac{P(S_1)}{|S_1|}\right]$, then
$\Fcal_{s,\lambda}\lr{\close{\frac{\lambda}{s}}{S}} \leq 
\Fcal_{s,\lambda}(S_1)$.}
\end{equation}
Indeed, if $\lambda\geq \lambda_2$ and $s\in
\left[0,1-\lambda\frac{P(S_1)}{|S_1|}\right]$,
then $s \leq 1-\lambda\frac{P(S_1)}{|S_1|} \leq \frac{\lambda}{R_2}$.
That is, $\frac{s}{\lambda} \leq \frac{1}{R_2}$. Then, by Lemma
\ref{le:FSL} \ref{it:FSL3b}, we have
\begin{equation}\label{botifarraY1}
P(\mathrm{Close}_{\frac{\lambda}{s}} (S)) + \frac{s}{\lambda}
|\mathrm{Close}_{\frac{\lambda}{s}} (S)\setminus S| \leq
 P(\mathrm{Close}_{R_2} (S)) + \frac{1}{R_2} |\mathrm{Close}_{R_2} (S)\setminus
S|
=   \frac{P(S_1)}{|S_1|}|S|.
\end{equation}
Now, $\Fcal_{s,\lambda}\lr{\close{\frac{\lambda}{s}}{S}} \leq 
\Fcal_{s,\lambda}(S_1)$ if and only if
$$
P(\mathrm{Close}_{\frac{\lambda}{s}} (S)) + \frac{s}{\lambda}
|\mathrm{Close}_{\frac{\lambda}{s}} (S)\setminus S| \leq
P(S_1)-   \frac{1-s}{\lambda}|S_1| + \frac{1-s}{\lambda}|S|
= P(S_1) +   \frac{1-s}{\lambda}|S_2|.
$$
Thus, by (\ref{botifarraY1}), it is sufficient to prove that
$$
\frac{P(S_1)}{|S_1|}|S|\leq  P(S_1) +   \frac{1-s}{\lambda}|S_2|.
$$
But this is true since $\frac{P(S_1)}{|S_1|} \leq \frac{1-s}{\lambda}$. Hence
(\ref{botifarraX2}) holds.

Since  $\Fcal_{s,\lambda}(S_1) < 0$ for $s\in
\left[0,1-\lambda\frac{P(S_1)}{|S_1|} \right)$, then
(\ref{botifarraX2}) implies
\begin{equation*}
\hbox{\rm if $\lambda\geq \lambda_2 = \frac{R_2|S_1|}{R_2P(S_1)+|S_1|}$ and 
$s\in  \left[0,1-\lambda\frac{P(S_1)}{|S_1|} \right)$, then
$\Fcal_{s,\lambda}\lr{\close{\frac{\lambda}{s}}{S}} < 0$.}
\end{equation*}
In particular, we have that $s_c(\lambda) > 1-\lambda\frac{P(S_1)}{|S_1|}$.

Observe also that the inequality in (\ref{botifarraX2}) is strict if
$s\in \left[0,1-\lambda\frac{P(S_1)}{|S_1|} \right)$.

\item 
Let us prove  that the optimum in $[0,1-\lambda\frac{P(S_1)}{|S_1|}]$ is
$\mathrm{Close}_{\frac{\lambda}{s}} (S)$. By (\ref{botifarraX2}) it cannot be
$S_1$. On the other hand,
by the first paragraph of Step b1, if $\lambda \in (\lambda_2,\lambda_3]$, 
we have that $\lambda > \lambda_2\geq \bar{\lambda}_1$ and, therefore,
$\frac{\lambda}{R_1} > 1-\lambda\frac{P(S_1)}{|S_1|}$. By the last remark in
Step (a1),
$\Fcal_{s,\lambda}\lr{\close{\frac{\lambda}{s}}{S}} \leq  \Fcal_{s,\lambda}(S)$
in the interval
$[0,1-\lambda\frac{P(S_1)}{|S_1|}]$. On the other hand, on that interval, 
$\Fcal_{s,\lambda}(S)\leq \Fcal_{s,\lambda}(S_2)$.
Thus, the optimum is $\mathrm{Close}_{\frac{\lambda}{s}} (S)$ (its energy being
negative).

Again, the optimum in $(1-\lambda\frac{P(S_1)}{|S_1|},s_c(\lambda)]$ is
$\mathrm{Close}_{\frac{\lambda}{s}} (S)$. The optimum cannot be $S_1$ because
its energy is positive.
By the assumption $r_2\leq r_1$, the energy of $S_2$ is also positive. Thus, also is for
$S$. The optimum is
$\mathrm{Close}_{\frac{\lambda}{s}} (S)$.

Notice that the argument in the previous paragraph shows that 
for $s>s_c(\lambda)$ the optimum can only be either
$\GammaS$ or $\emptyset$. Thus, either there is a maximal interval of values of
$s$, say
$(s_c(\lambda),s_d(\lambda)]$, not reduced to $s_c(\lambda)$,
where (\ref{eq:S2}) holds, in which case $\mathrm{Close}_{\frac{\lambda}{s}}
(S)$ is the optimum up to $s_d(\lambda)$, or
the energy of $\mathrm{Close}_{\frac{\lambda}{s}} (S)$  becomes positive
immediately after $s_c(\lambda)$ and the optimum is
$\emptyset$. Thus, we may take $s_c(\lambda)$ as the maximal solution of
(\ref{eq:S2}).
\end{enumerate}

\item[(c2)] 
{
\it 
Assume that $ \fracco \leq \fraceins$.}

\noindent In this case,
$\lambda_2=\frac{|S_2|}{P(\mathrm{co}(S)) - P(S_1)}$.
Since $P(\mathrm{co}(S)) < P(S)$, we have $\Fcal_{0,\lambda}(\mathrm{co}(S)) <
\Fcal_{0,\lambda}(S)$.
If we take $\lambda > \lambda_2 $, we have $\Fcal_{0,\lambda}(S_1) <
\Fcal_{0,\lambda}(\mathrm{co}(S))$.
Thus, for $s\in  \left[ 0,1-\lambda\frac{P(S_1)}{|S_1|} \right]$ small enough
the optimum is $S_1$. By monotonicity of the level sets of $u_\lambda$, $S_1$ is
the optimum for all
$s\in \left[0,1-\lambda\frac{P(S_1)}{|S_1|} \right]$.
For $s > 1-\lambda\frac{P(S_1)}{|S_1|}$ is the emptyset.

\end{enumerate}

\item[(d)] Since $\lambda>\lambda_3 = \norm{\chi_S}_{*}$, 
$\Cs = \emptyset$ by Proposition \ref{pro:gnorm}. 
This concludes the proof.
\end{enumerate}
\qed

\begin{figure}
 \includegraphics[]{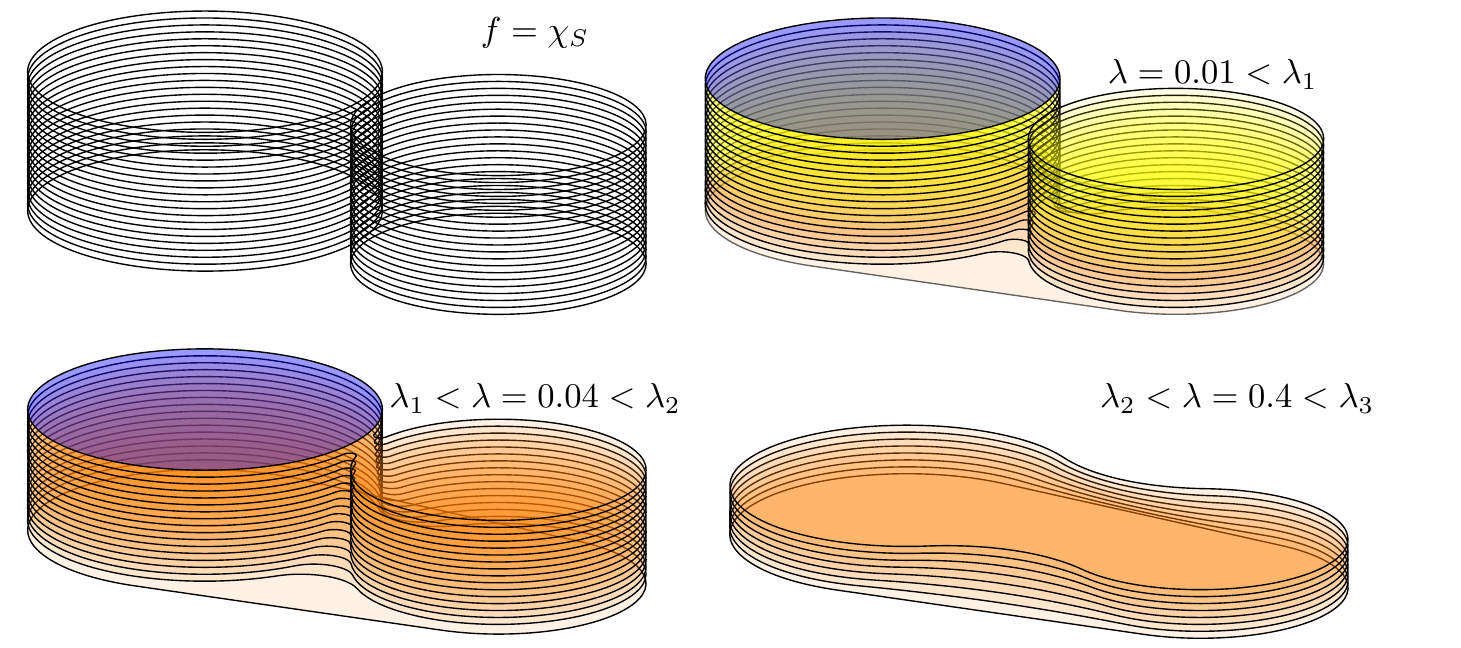}
\caption{Levelsets of the solutions $u_\lambda$ for different values of
$\lambda$. 
$S_1$, $S_2$ are balls with radius $r_1=1.2,r_2=1$, 
the distance between them is $d=0.05$. In this case
$\fraceins < \fracco$. 
}
\label{fi:solution}
\end{figure}

\end{document}